\documentclass[11pt]{article}%
\usepackage{color}
\usepackage{epsfig}
\usepackage{amsmath,amssymb}
\usepackage{theorem,                                             amsfonts,
pifont}
\usepackage{amsmath}
\usepackage{amsfonts}
\usepackage{amssymb}
\usepackage{graphicx}%
\newcommand{\Lie}{\mathrm{Lie}}
\newcommand{\LARC}{\mathrm{LARC}}

\newcommand{\diag}{\mathrm{diag}}
\newcommand{\Tr}{\mathrm{Tr}}

\newcommand{\Id}{\mathrm{Id}}

\newcommand{\R}{\mathbb{R}}
\newcommand{\RP}{\mathbb{RP}}

\newcommand{\C}{\mathbb{C}}
\newcommand{\beq}{\begin{equation}}
\newcommand{\eeq}{\end{equation}}
\newcommand{\be}{\begin{equation}}
\newcommand{\ee}{\end{equation}}
\newcommand{\bea}{\begin{eqnarray}}
\newcommand{\eea}{\end{eqnarray}}
\newcommand{\br}{\begin{eqnarray}}
\newcommand{\er}{\end{eqnarray}}
\newcommand{\brs}{\begin{eqnarray*}}
\newcommand{\ers}{\end{eqnarray*}}
\newcommand{\ba}{\begin{array}}
\newcommand{\ea}{\end{array}}

\newcommand{\bed}{\begin{description}}
\newcommand{\eed}{\end{description}}

\newenvironment{proof}{\noindent {\em Proof. }}{\hfill $\blacksquare$ \vskip 3pt}

\newtheorem{thm}{Theorem}[section]
\newtheorem{theorem}[thm]{Theorem}
\newtheorem{lemma}[thm]{Lemma}
\newtheorem{cor}[thm]{Corollary}
\newtheorem{prop}[thm]{Proposition}
\begin{theorembodyfont}{\rmfamily}

\newtheorem{remark}[thm]{Remark}

\end{theorembodyfont}
\newtheorem{deff}[thm]{Definition}

\newtheorem{proposition}[thm]{Proposition}
\begin{document}

\title{Growth rates for persistently excited linear systems }
\author{Yacine Chitour \thanks{ Laboratoire des Signaux et Syst\`emes, Sup\'elec, Gif
s/Yvette, France and Universit\'e Paris Sud, Orsay and Team GECO, INRIA
Saclay-\^Ile-de-France, \texttt{chitour@lss.supelec.fr}} \qquad Fritz Colonius
\thanks{ Institut f\"ur Mathematik, Universit\"at Augsburg, Augsburg, Germany,
\texttt{fritz.colonius@math.uni-augsburg.de}} \qquad Mario Sigalotti \thanks{
INRIA Saclay-\^Ile-de-France, Team GECO, and CMAP, UMR 7641, \'Ecole
Polytechnique, Palaiseau, France, \texttt{mario.sigalotti@inria.fr} } }
\date{}
\maketitle

\begin{abstract}
We consider a family of linear control systems $\dot{x}=Ax+\alpha Bu$ where
$\alpha$ belongs to a given class of persistently exciting signals. We seek
maximal $\alpha$-uniform stabilisation and destabilisation
by means of linear feedbacks $u=Kx$.
We extend previous results
obtained for bidimensional single-input linear control systems to the general
case as follows: if the pair $(A,B)$ verifies a certain Lie bracket generating
condition, then the maximal rate of convergence of $(A,B)$ is equal to the
maximal rate of divergence of $(-A,-B)$. We also provide more precise results
in the general single-input case, where the above result is obtained under the
sole assumption of controllability of the pair $(A,B)$.

\end{abstract}



\section{Introduction\label{Section1}}

In the present paper we address stabilization issues relative to linear
systems subject to scalar persistently exciting signals (PE-signals). Such a
linear time-dependent system is written as
\begin{equation}
\dot{x}=Ax+\alpha(t)Bu\,, \label{system}%
\end{equation}
where $x\in\mathbb{R}^{d}$, $u\in\mathbb{R}^{m}$, the matrices $A,B$ have
appropriate sizes and the function $\alpha$ is a \emph{scalar} PE-signal,
i.e., $\alpha$ takes values in $[0,1]$ and
there exist two positive constants $\mu,T$ such that, for every $t\geq0$,
\begin{equation}
\int_{t}^{t+T}\alpha(s)ds\geq\mu. \label{Tmu}%
\end{equation}
Given two positive real numbers $\mu\leq T$, we use $\mathcal{{G}}(T,\mu)$ to
denote the class of all PE signals verifying \eqref{Tmu}.

In \eqref{system}, the PE-signal $\alpha$ can be seen as an input perturbation
modelling the fact that the instants where the control $u$ acts on the system
are not exactly known. If $\alpha$ only takes the values $0$ and $1$, then
\eqref{system} actually switches between the uncontrolled system $\dot x=Ax$
and the controlled one $\dot x = Ax+Bu$. In that context, the persistence of
excitation condition \eqref{Tmu} is designed to guarantee some action on the
system. Persistent of excitation conditions have appeared both in the
identification and in the control literatures
\cite{Anderson1977Exponential,Anderson1986Stability,Andersson2002Degenerate,Brockett2000Rate,Loria2005PE,Lovera2004Global,Morgan1977Stability}%
.

Here, we are mainly concerned with the global asymptotic stabilization of
system \eqref{system} with a constant linear feedback $u=Kx$ \emph{uniformly}
with respect to all PE-signals $\alpha\in\mathcal{{G}}(T,\mu)$. The dual
problem consists in exponentially destabilizing system \eqref{system} by a
constant linear feedback. In order to quantitatively measure these
stabilization and destabilization features, we first define, for every $K$,
the exponential rate of convergence for the family of time varying-systems
$\dot{x}=(A+\alpha BK)x$ and use $\mathrm{rc}(A,B,K)$ to denote it. Similarly,
for every $K$, let $\mathrm{rd}(A,B,K)$ be the rate of divergence for the
family of time varying-systems. (For the precise definitions of $\mathrm{rc}%
(A,B,K)$ and $\mathrm{rd}(A,B,K)$ in terms of Lyapunov exponents, see
Section~\ref{def--rc}.)
The sign convention on $\mathrm{rc}(A,B,K)$ (respectively, $\mathrm{rd}%
(A,B,K)$) is such that exponential stabilizability (respectively,
destabilizability) of system \eqref{system} is equivalent to the existence of
some feedback $K$ with $\mathrm{rc}(A,B,K)>0$ (respectively, $\mathrm{rd}%
(A,B,K)>0$). If $K$ is such that $\mathrm{rc}(A,B,K)>0$ then we say that $K$
is a $(T,\mu)$-stabilizer. Let $\mathrm{RC}(A,B)$ and $\mathrm{RD}(A,B)$ be
defined as the supremum over $K$ of $\mathrm{rc}(A,B,K)$ and $\mathrm{rd}%
(A,B,K)$ respectively.

Recall that if $T=\mu$ then $\alpha\equiv1$ is the unique choice of PE-signal
and in that case the above issues correspond to the classical stabilizability
questions associated with time-invariant finite-dimensional linear control
systems $\dot x=Ax+Bu$. In particular, it follows from the pole-shifting
theorem that $\mathrm{RC}(A,B)=+\infty$ if and only if $\mathrm{RD}%
(A,B)=+\infty$, and this happens if and only if the pair $(A,B)$ is controllable.

The present paper belongs to a line of research initiated in \cite{MCSS} which
consists in generalizing the pole-shifting theorem to linear control systems
subject to persistence of excitation on the input (for a survey on recent
results on persistence of excitation, see \cite{survey2013}). The
pole-assignment part of that theorem seems difficult to transpose in the
context of persistence of excitation. Therefore, we are more interested in a
qualitative feature that we call \emph{generalized pole-shifting property},
namely whether $\mathrm{RC}(A,B)$ and $\mathrm{RD}(A,B)$ are both infinite and
to characterize such a property in terms of the data of the problem
$A,B,T,\mu$.

When $\mu<T$, the generalized pole-shifting property
is not guaranteed.
More precisely, it has been proved in \cite{ChitourSigalotti2010} that for
bidimensional single-input controllable systems of the form \eqref{system},
there exists $\rho\in(0,1)$ (independent of $A,B$) such that if $\mu/T<\rho$
then $\mathrm{RC}(A,B)$ is finite. As a consequence, one easily deduces that
for $\lambda$ large enough $\mathrm{rc}(A+\lambda\mathrm{Id},B,K)$ is negative
for every $K$, hence there does not even exist a $(T,\mu)$-stabilizer in that
situation. Let us mention that if one restricts
$\mathcal{G}(T,\mu)$ to the subclass $\mathcal{D}(T,\mu,M)$ of its elements
which are $M$-Lipschitz for a given $M>0$, then one recovers that, for every
$0<\mu<T$, system \eqref{system} can be stabilized and destabilized with
arbitrarily large exponential rates uniformly with respect to $\alpha
\in\mathcal{D}(T,\mu,M)$ (cf.~\cite{Mazanti}).

Our main goal in this paper is to relate the maximal rates of convergence and
divergence associated with the pairs $(A,B)$ and $(-A,-B)$. Recall that in the
case $T=\mu$, one trivially has that $\mathrm{RD}(A,B)=\mathrm{RC}(-A,-B)$. On
the other hand, it was proved in \cite{ChitourSigalotti2010} that
$\mathrm{RC}(A,B)=+\infty$ if and only if $\mathrm{RD}(A,B)=+\infty$ for
bidimensional single-input controllable systems of the form \eqref{system}.
The main result we obtain in this paper is Theorem \ref{5p1}. It shows that
the maximal rate of convergence for a persistently excited system coincides
with the maximal rate of divergence for the time-reversed system, provided
that 
there exists a feedback
$K$ 
such that $\mathrm{Lie}(A-(\Tr(A)/d)\Id_d, BK-(\Tr(BK)/d)\Id_d)$, the Lie algebra generated by $A-(\Tr(A)/d)\Id_d$
and $BK-(\Tr(BK)/d)\Id_d$ is equal to $\mbox{sl}(d,\R)$.  If $d\geq 3$, we slightly simplify the latter condition by merely asking that 
there exists a feedback $K$ such that $\mathrm{Lie}(A, BK)$ is equal to $\mbox{gl}(d,\R)$. To prove that result, we first
prove  that $\mathrm{RD}(A,B)=\mathrm{RC}(-A,-B)$ if there exists a feedback
$K$ such that the projection on the real projective space
$\mathbb{RP}^{d-1}$ of the bilinear system $\dot x=Ax+v BKx$, $x\in
\mathbb{R}^{d}$, $v\in\mathbb{R}$, satisfies the Lie algebraic rank condition.
We denote by $\mathrm{PLARC}(A,B)$ the set of all such $K$. 
In the single-input case, we can refine the result by showing that if $(A,B)$
is controllable then $\mathrm{PLARC}(A,B)$ is nonempty 
(and conversely if
$d\geq3$). Moreover, we show in a second step that nonemptiness of $\mathrm{PLARC}(A,B)$
is actually equivalent to nonemptiness of $\mathrm{LARC}_0(A,B)$, i.e., the set of feedbacks $K$
such that $\mathrm{Lie}(A-(\Tr(A)/d)\Id_d, BK-(\Tr(BK)/d)\Id_d)$ is equal to $\mbox{sl}(d,\R)$.  

Let us briefly describe the techniques used in the paper. In order to relate
asymptotic properties of \eqref{system} and of the corresponding time-reversed
system, one must take advantage of the linearity of the problem by analyzing
the periodic trajectories of the projected control system on $\mathbb{RP}%
^{d-1}$. Thus, one is naturally led to consider the family of continuous
linear flows on a vector bundle defined by the persistently excited systems
associated with the feedbacks $K$. Such constructions have been used for
bilinear control systems in Colonius and Kliemann \cite{ColKli} and for
switched systems by Wirth in \cite{Wirth2005}. The crucial technical step
consists of extending to the PE context the results of \cite{ColKli} asserting
that if $K\in\mathrm{PLARC}(A,B)$ then periodic trajectories of the projected
system corresponding to periodic PE-signals retain the asymptotic properties
of the original system. Finally, since our rates of convergence/divergence are
defined in terms of Lyapunov exponents, we rely on tools from dynamical
systems theory such as Morse spectrum and control sets, which are used for
proving regularity properties for the functions $(A,B,K)\mapsto\mathrm{rc}%
(A,B,K)$ and $(A,B,K)\mapsto\mathrm{rd}(A,B,K)$.

Let us mention the recent contribution to the theory of linear control systems
with general time-varying coefficients given by Anderson, Ilchmann and Wirth
\cite{AndIW13}, also based on Lyapunov exponents. Our contribution is
independent of their results, since persistently excited systems present
distinctive features.

Before providing the structure of the paper, let us open some perspectives for
future work related to the issues discussed here. First of all, it would be
interesting to relate, in the multi-input case, the nonemptiness of $\mathrm{LARC}_0(A,B)$
with algebraic properties of the pair $(A,B)$. Secondly, the understanding of the 
generalized pole-shifting property for $d\geq3$ and in the multi-input case is still a widely open problem.

The contents of this paper are as follows: Section \ref{Section2} provides the
notion of persistently excited system as well as growth rates of solutions. In
particular, the maximal rates of convergence and divergence are defined.
Furthermore, Lie algebraic conditions are recalled for bilinear control
systems in $\mathbb{R}^{d}$ and their projections onto projective space.
Section \ref{Section3} shows that the exponential growth rates can
be determined by certain periodic trajectories of the projected systems.
This is used in Section \ref{Section4} to derive continuity properties of
growth rates. In Section~\ref{s-5} the relation between maximal rates of
convergence and divergence is explored and the main result is given and
commented (Theorem~\ref{5p1}). Finally, Section \ref{Section5} gives a
detailed analysis of the single-input case.

{\bf Acknowledgements} It is a pleasure to acknowledge U.~Helmke and P.~Kokkonen 
for 
pointing out,  respectively, the papers \cite{dirr-helmke,volklein} and \cite{amayo,kissin}, which led us to 
Proposition~\ref{ptrivial}. We also thank J-P.~Gauthier and F.~Wirth for several fruitful exchanges. 

\section{Problem formulation and preliminaries\label{Section2}}

\label{def00}

In this section we introduce formally persistently excited linear systems and
recall notion and facts concerning their stability properties. In particular,
Lyapunov exponents and associated rates of convergence and divergence are
recalled. Finally, accessibility properties of related control systems are discussed.

\subsection{PE systems and $(T,\mu)$-stabilizers}

The following notion is fundamental for this paper.

\begin{deff}
[$(T,\mu)$-signal]\label{Tm-signal} Let $\mu$ and $T$ be positive constants
with $\mu\leq T$. A \emph{$(T,\mu)$-signal} is a measurable function
$\alpha:\mathbb{R}\rightarrow\lbrack0,1]$ satisfying
\begin{equation}
\int_{t}^{t+T}\alpha(s)ds\geq\mu\,\text{\ for all }t\in\mathbb{R}\,.
\label{EP}%
\end{equation}
We use $\mathcal{{G}}(T,\mu)$ to denote the set of all $(T,\mu)$-signals.

\end{deff}

Given two positive integers $d\geq 2$ and $m\geq 1$, let $M_{d,m}(\mathbb{R})$ be the set of $d\times m$
matrices with real entries and we use $M_{d}(\mathbb{R})$ to denote $M_{d,
d}(\mathbb{R})$. We write $P_{d,m}$ for $M_{d}(\mathbb{R})\times M_{d,
m}(\mathbb{R})$.

\begin{deff}
[PE system]\label{Tm-sys} Given two positive constants $\mu$ and $T$ with
$\mu\leq T$ and a
pair $(A,B)\in P_{d,m}$, we define the \emph{persistently excited system} (PE
system for short) associated with $T,\mu,A$, and $B$ as the family of
non-autonomous linear control systems
\begin{equation}
\label{system1}\dot x=Ax+\alpha B u, \ \ \ \alpha\in\mathcal{{G}}(T,\mu).
\end{equation}

\end{deff}

Given a persistently excited system \eqref{system1}, we consider the following
problem: Is it possible to stabilize \eqref{system1} \emph{uniformly} with
respect to every $(T,\mu)$-signal $\alpha$, i.e., to find a matrix $K\in
M_{m,d}(\mathbb{R})$ which makes the origin globally asymptotically stable
for
\begin{equation}
\dot{x}=(A+\alpha(t)BK)x, \label{feedback}%
\end{equation}
with $K$ depending only on $A$, $B$, $T$ and $\mu$?

Note that (\ref{feedback}) defines a linear continuous flow $\Phi$ on the
vector bundle $\mathcal{{G}}(T,\mu)\times\mathbb{R}^{d}$, since $\mathcal{{G}%
}(T,\mu)$ is a shift-invariant (i.e., $\alpha(\cdot)$ is a $(T,\mu)$-signal if
and only if the same is true for $\alpha(t_{0}+\cdot)$ for every $t_{0}%
\in\mathbb{R}$), convex and weak-$\star$ compact
subset of $L^{\infty}(\mathbb{R},\mathbb{R})$ (see \cite{ColKli} for definitions).

Referring to $
x(\cdot\,;t_{0},x_{0},A,B,K,\alpha)
$
as the solution of \eqref{feedback} passing through $x_{0}$ at time $t_{0}$,
we introduce the following definition.

\begin{deff}
[$(T,\mu)$-stabilizer]\label{stab} Let $T\geq\mu>0$. The gain $K\in
M_{m,d}(\mathbb{R})$ is said to be a \emph{$(T,\mu)$-stabilizer} for
\eqref{system1} if \eqref{feedback} is globally exponentially stable,
uniformly with respect to $\alpha\in\mathcal{{G}}(T,\mu)$, i.e., there exist
$C,\gamma>0$ such that every solution $x(\cdot\,;t_{0},x_{0},A,B,K,\alpha)$ of
\eqref{feedback} satisfies
\[
|x(t;t_{0},x_{0},A,B,K,\alpha)|\leq Ce^{-(t-t_{0})\gamma}|x_{0}|\text{\ for
every }t\geq t_{0}.
\]
\end{deff}

The definition above is clearly independent of the choice of the norm on
$\mathbb{R}^{d}$. In the following, we assume $|\cdot|$ to be a fixed norm in
$\mathbb{R}^{d}$ and we denote by $\Vert\cdot\Vert$ the induced matrix norm.

\begin{remark}
Since $\mathcal{{G}}(T,\mu)$ is shift-invariant and compact, Fenichel's
uniformity lemma (see \cite[Lemma 5.2.7]{ColKli}) allows one to restate
equivalently the above definition in the following weaker form: $K\in
M_{m,d}(\mathbb{R})$ is a \emph{$(T,\mu)$-stabilizer} for \eqref{system1} if,
for every $\alpha\in\mathcal{{G}}(T,\mu)$, every solution $x(\cdot
\,;t_{0},x_{0},A,B,K,\alpha)$ of \eqref{feedback} tends to zero as time goes
to $+\infty$.
\end{remark}

\subsection{Convergence and divergence rates and a generalized pole-shifting
property}

\label{def--rc}

Next we introduce a number of rates describing the stability properties of PE
systems. Let $(A,B)\in P_{d,m}$, $K$ belong to $M_{m,d}(\mathbb{R})$ and
$T\geq\mu>0$. For $\alpha\in\mathcal{{G}}(T,\mu)$ and $0\not =x_{0}%
\in\mathbb{R}^{d}$
let
\begin{align*}
\lambda^{+}(x_{0},A,B,K,\alpha)  &  =\limsup_{t\rightarrow+\infty}\frac{1}%
{t}\log|x(t;0,x_{0},A,B,K,\alpha)|,\\
\lambda^{-}(x_{0},A,B,K,\alpha)  &  =\liminf_{t\rightarrow+\infty}\frac{1}%
{t}\log|x(t;0,x_{0},A,B,K,\alpha)|.
\end{align*}
Set
\[
\Lambda^{+}(A,B,K,\alpha)=\sup_{x_{0}\not =0}\lambda^{+}(x_{0},A,B,K,\alpha
),\qquad\Lambda^{-}(A,B,K,\alpha)=\inf_{x_{0}\not =0}\lambda^{-}%
(x_{0},A,B,K,\alpha).
\]
The \textit{rate of convergence} and the \textit{rate of divergence}
associated with the family of systems $\dot{x}=(A+\alpha BK)x$, $\alpha
\in\mathcal{{G}}(T,\mu)$, are defined as
\begin{equation}
\mathrm{rc}(A,B,K)=\inf_{\alpha\in\mathcal{{G}}(T,\mu)}\big(-\Lambda
^{+}(A,B,K,\alpha)\big)\text{ and }\mathrm{rd}(A,B,K)=\inf_{\alpha
\in\mathcal{{G}}(T,\mu)}\Lambda^{-}(A,B,K,\alpha)), \label{td0}%
\end{equation}
respectively. In particular $\mathrm{rc}(A,B,K)>0$ if and only if $K$ is a
$(T,\mu)$-stabilizer for \eqref{system1}. Notice that, differently from
\cite{ChitourSigalotti2010}, we omit here from the arguments of $\mathrm{rc}$
and $\mathrm{rd}$ the quantities $T,\mu$ (on which they actually depend),
since we focus here
on the dependence of these objects on $A$, $B$, and $K$.

Since each signal constantly equal to $\bar{\alpha}\in\lbrack\mu/T,1]$ is in
$\mathcal{{G}}(T,\mu)$, one immediately gets the estimates
\begin{equation}
\mathrm{rc}(A,B,K)\leq\min_{\bar{\alpha}\in\lbrack\mu/T,1]}\min(-\Re
(\sigma(A+\bar{\alpha}BK))), \label{td-vp}%
\end{equation}
and
\begin{equation}
\mathrm{rd}(A,B,K)\leq\min_{\bar{\alpha}\in\lbrack\mu/T,1]}\min(\Re
(\sigma(A+\bar{\alpha}BK))), \label{td-vp1}%
\end{equation}
where $\sigma(M)$ denotes the spectrum of a matrix $M$ and $\Re(\zeta)$ the
real part of a complex number $\zeta$.

A linear change of coordinates $y=Px$, $v=Vu$ does neither affect $\Lambda
^{+}(A,B,K,\alpha)$ nor $\Lambda^{-}(A,B,K,\alpha)$. Hence
\begin{equation}
\mathrm{rc}(A,B,K)=\mathrm{rc}(PAP^{-1},PBV^{-1},VKP^{-1}), \label{chco1}%
\end{equation}
and
\begin{equation}
\mathrm{rd}(A,B,K)=\mathrm{rd}(PAP^{-1},PBV^{-1},VKP^{-1}), \label{chco}%
\end{equation}
for all invertible matrices $P\in M_{d}(\mathbb{R})$ and $V\in M_{m}%
(\mathbb{R})$.

\begin{remark}
\label{r-con} Let $P$ be a change of coordinates which brings the pair $(A,B)$
into a controllability decomposition $(A^{\prime},B^{\prime})$ of $(A,B)$,
namely,
\begin{equation}
\label{contr-dec}A^{\prime}=%
\begin{pmatrix}
A_{1} & A_{2}\\
0 & A_{3}%
\end{pmatrix}
,\quad B^{\prime}=%
\begin{pmatrix}
B_{1}\\
0
\end{pmatrix}
,
\end{equation}
with $(A_{1},B_{1})$ controllable. Then, by \eqref{chco1}, \eqref{chco}, and a
standard argument based on the variation of constant formula, one gets that,
for every $K=(K_{1}\ K_{2})\in M_{d, m}(\mathbb{R})$,
\begin{align}
\label{rc-cont}\mathrm{rc}(A,B,K)  &  =\min(\mathrm{rc}(A_{1},B_{1}%
,K_{1}),\min(-\Re(\sigma(A_{3})))),\\
\mathrm{rd}(A,B,K)  &  =\min(\mathrm{rd}(A_{1},B_{1},K_{1}),\min(\Re
(\sigma(A_{3})))). \label{rd-cont}%
\end{align}

\end{remark}

Define the \emph{maximal rate of convergence} associated with the PE system
\eqref{system1} as
\begin{equation}
\mathrm{RC}(A,B)=\sup_{K\in M_{m,d}(\mathbb{R})}\mathrm{rc}(A,B,K),
\label{tdd}%
\end{equation}
and similarly, the \emph{maximal rate of divergence} as
\begin{equation}
\mathrm{RD}(A,B)=\sup_{K\in M_{m,d}(\mathbb{R})}\mathrm{rd}(A,B,K).
\label{tDD}%
\end{equation}
Because of \eqref{chco1} and \eqref{chco}, one has
\begin{equation}
\mathrm{RC}(A,B)=\mathrm{RC}(PAP^{-1},PBV^{-1}),\qquad\mathrm{RD}%
(A,B)=\mathrm{RD}(PAP^{-1},PBV^{-1}), \label{chch}%
\end{equation}
for all invertible matrices $P\in M_{d}(\mathbb{R})$ and $V\in M_{m}%
(\mathbb{R})$.

Thanks to Remark~\ref{r-con}, one deduces for a controllability decomposition
of the pair $(A,B)$ as in \eqref{contr-dec} that
\begin{equation}
\mathrm{RC}(A,B)=\min(\mathrm{RC}(A_{1},B_{1}),\min(-\Re(\sigma(A_{3}%
)))),\quad\mathrm{RD}(A,B)=\min(\mathrm{RD}(A_{1},B_{1}),\min(\Re(\sigma
(A_{3})))). \label{RCRD-cont}%
\end{equation}
Notice also that
\begin{equation}
\mathrm{RC}(A+\lambda\mathrm{Id}_{d},B)=\mathrm{RC}(A,B)-\lambda
,\qquad\mathrm{RD}(A+\lambda\mathrm{Id}_{d},B)=\mathrm{RD}(A,B)+\lambda.
\label{rem2}%
\end{equation}

\begin{remark}
Let $(A,B)\in P_{d,m}$ for some $d,m\in\mathbb{N}$. A necessary condition for
one of the quantities $\mathrm{RC}(A,B)$ or $\mathrm{RD}(A,B)$ to be infinite
is that the pair $(A,B)$ is controllable. This immediately follows from
\eqref{RCRD-cont}.

\end{remark}

\begin{remark}
\label{2.6} Let $m=1$ and suppose that for $A$ there exists $\bar{B}$ for
which $(A,\bar{B})$ is controllable. Then $\mathrm{RC}(A,B)$ and
$\mathrm{RD}(A,B)$ do not depend on $B$, as long as $(A,B)$ is controllable.
This follows from
\eqref{chch} and the fact that the controllability form of a single-input
controllable system only depends on the matrix $A$.
\end{remark}

Given a controllable pair $(A,B)$, whether or not $\mathrm{RC}$ and
$\mathrm{RD}$ are both infinite can be understood as whether or not a
\emph{generalized pole-shifting property} holds true for the PE system
$\dot{x}=Ax+\alpha Bu$, $\alpha\in\mathcal{{G}}(T,\mu)$. One of the aims of
the paper is to investigate up to which extent the unboundedness of
$\mathrm{RC}$ and $\mathrm{RD}$ are equivalent properties. In the planar
single-input case the two properties are equivalent, as recalled below
(\cite[Proposition 4.3]{ChitourSigalotti2010}).

\begin{prop}
\label{rdc0} Let $d=2$ and $m=1$ and consider a PE system of the form $\dot
{x}=Ax+\alpha Bu$, $\alpha\in\mathcal{{G}}(T,\mu)$.
Then $\mathrm{RC}(A,B)=+\infty$ if and only if $\mathrm{RD}(A,B)=+\infty$.
\end{prop}

\subsection{Projected dynamics on $\mathbb{RP}^{d-1}$}

For a matrix $A\in M_{d}(\mathbb{R})$, we denote by $\Pi A$ the vector field
on the real projective space $\mathbb{RP}^{d-1}$ obtained by canonical
projection of the vector field $x\mapsto Ax$ onto $T\mathbb{RP}^{d-1}$, i.e.,
for every $q=\Pi x\in\mathbb{RP}^{d-1}$ with $x\in\mathbb{R}^{d}%
\setminus\{0\}$,
\[
(\Pi A)(q)=d\Pi_{x}\left(  Ax-\frac{\langle x,Ax\rangle}{|x|^{2}}x\right)  .
\]
Notice that $\Pi A=\Pi(A+\lambda \Id)$ for every $\lambda\in\R$.

Given two matrices $A_{1}$ and $A_{2}$ in $M_{d}(\mathbb{R})$ and a set of
admissible controls $\mathcal{U}\subset L^{\infty}(\mathbb{R},[0,1])$, we
define three control systems as follows:
\begin{align}
\dot x  &  =A_{1}x+u A_{2}x, & x\in\mathbb{R}^{d},\quad u\in\mathcal{U}%
,\label{u-s}\\
\dot q  &  =(\Pi A_{1})(q)+u (\Pi A_{2})(q), & q\in\mathbb{RP}^{d-1},\quad
u\in\mathcal{U},\label{p-s}\\
\dot M  &  =A_{1}M+u A_{2}M, & M\in M_{d}(\mathbb{R}),\quad u\in\mathcal{U}.
\label{m-s}%
\end{align}

We say that $\{A_{1},A_{2}\}$ satisfies the \emph{Lie algebra rank condition}
if the Lie algebra $\mathrm{Lie}(A_{1},A_{2})$ generated by $A_{1}$ and
$A_{2}$
is equal to $M_{d}(\mathbb{R})$. Similarly, we say that $\{A_{1},A_{2}\}$
satisfies the \emph{projected Lie algebra rank condition } if $\{\Pi A_{1},\Pi
A_{2}\}$ satisfies the Lie algebra rank condition on $\mathbb{RP}^{d-1}$,
i.e., $\mathrm{Lie}_{q}(\Pi A_{1},\Pi A_{2})=T_{q}\mathbb{RP}^{d-1}$ for every
$q\in\mathbb{RP}^{d-1}$. This coincides with hypothesis (H) in
\cite{CK-article}.

Given a pair $(A,B)\in P_{d,m}$, let $\mathrm{LARC}(A,B)$ (respectively,
$\mathrm{PLARC}(A,B)$) be the set of
$K\in M_{m, d}(\mathbb{R})$ such that $\{A,B K\}$ satisfies the Lie algebra
rank condition (respectively, the projected Lie algebra rank condition).
We also find useful to introduce the set $\mathrm{LARC}_0(A,B)$ of $M_{m, d}(\mathbb{R})$ made of those feedbacks $K$ such that 
$\mathrm{Lie}(A-(\Tr(A)/d)\Id_d,BK-(\Tr(BK)/d)\Id_d)=\mbox{sl}(d,\R)$.

The proof of the following lemma is trivial.

\begin{lemma}
\label{proj} Let $A_{1},A_{2}\in M_{d}(\mathbb{R})$. Then $\Pi\lbrack
A_{1},A_{2}]=[\Pi A_{1},\Pi A_{2}]$. As a consequence, the attainable set for
\eqref{p-s} from every initial condition $q_{0}=\Pi x_{0}\in\mathbb{RP}^{d-1}%
$, is the projection on $\mathbb{RP}^{d-1}$ of the attainable set of
\eqref{u-s} from $x_{0}$ and the evaluation at $q_{0}$ of the attainable set
for \eqref{m-s} from the identity. Moreover, for every $(A,B)\in P_{d,m}$ and
$\lambda\in\mathbb{R}$,
\begin{equation}
\label{comparisons}
\mathrm{LARC}(A+\lambda\mathrm{Id},B)\subseteq
\mathrm{LARC}_0(A,B)\subseteq \mathrm{PLARC}(A,B) .
\end{equation}

\end{lemma}


\begin{remark}
\label{irre} For $K\in M_{m,d}(\mathbb{R})$, define the \emph{system group}
$G_{K}$ (respectively, $G^0_K$) as the orbit through the identity for system \eqref{m-s}, with $A_1=A$ and $A_2=BK$ (respectively, $A_1=A-(\Tr(A)/d)\Id_d$ and $A_2=BK-(\Tr(BK)/d)\Id_d$). It is well
known that $G_{K}$ and $G^0_K$ are Lie subgroups of $M_{d}(\mathbb{R})$ with Lie algebras
given by $\mathrm{Lie}(A,BK)$ and $\mathrm{Lie}(A-(\Tr(A)/d)\Id_d,BK-(\Tr(BK)/d)\Id_d)$, respectively.
The actions of  $G_{K}$ and $G^0_{K}$ on $\mathbb{RP}^{d-1}$ coincide. 
Moreover, by the orbit theorem applied to the analytic
system \eqref{p-s}, 
such an action is transitive if and only if 
$K$ is in $\mathrm{PLARC}(A,B)$. {
(For details, see \cite{Jur}.)
}
In particular, if $K$ is in $\mathrm{PLARC}(A,B)$ then the Lie algebra 
$\mathrm{Lie}(A,BK)$ is irreducible, i.e., there does not exist a proper
subspace of $\mathbb{R}^{d}$
which is invariant for all the elements of $\mathrm{Lie}(A,BK)$.
\end{remark}

\section{Growth rates and periodicity\label{Section3}}

We start this section by a controllability property for the induced system on
projective space, which is useful for the subsequent discussion on growth rates.

Let us consider, for a moment, the system
\begin{equation}
\dot{x}=(A+v(t)BK)x, \label{cs_1}%
\end{equation}
where $K\in M_{m,d}(\mathbb{R})$ is a given feedback matrix and the role
previously played by the $(T,\mu)$-signal $\alpha$ is now taken by $v$, seen
as a control parameter, with values in a closed subinterval $I$ of $[0,1]$
with nonempty interior (the \emph{control range}). We assume that $v$ to
belong to $L^{\infty}(\mathbb{R},I)$, without persistent excitation
assumptions on it.

As noticed in Section \ref{def00}, the homogeneous bilinear control system
\eqref{cs_1} in $\mathbb{R}^{d}$ induces a control system in the projective
space $\mathbb{RP}^{d-1}$,
given by
\addtocounter{equation}{1} \newcounter{cs_2} \setcounter{cs_2}{\theequation}
\[
\dot{q}=(\Pi A)(q)+v(t)(\Pi BK)(q).\eqno{(\arabic{cs_2})_I}
\]
Denote by $t\mapsto q(t;q_{0},v)$ the trajectory of $(\arabic{cs_2})_{I}$ with
initial condition $q(0)=q_{0}\in\mathbb{RP}^{d-1}$ corresponding to the
control $v\in L^{\infty}(\mathbb{R},I)$.

The following controllability property motivates the role of the assumption
that the feedback matrix $K$ is in $\mathrm{PLARC}(A,B)$, which will appear
repeatedly in the following sections.

\begin{theorem}
\label{uniform-time} Consider the projected system $(\arabic{cs_2})_{I}$,
where $I\subset\lbrack0,1]$ is a closed interval with nonempty interior and
$K\in\mathrm{PLARC}(A,B)$. Then there exists a unique compact subset $C$ of
$\mathbb{RP}^{d-1}$ with nonempty interior having the following properties:

\begin{description}
\item {(i)} For all $q_{0}\in C$, $t\geq0$, and $v\in L^{\infty}%
(\mathbb{R},I)$ one has $q(t;q_{0},v)\in C$.

\item {(ii)} For every $q_{-}\in\mathrm{int}C$ there exists a time $\hat{\tau
}>0$ such that for all $q_{0}\in\mathbb{RP}^{d-1}$ there is $v_{0}\in
L^{\infty}(\mathbb{R},I)$ with%
\[
q(\tau;q_{0},v)=q_{-}\text{ for some }\tau\in\lbrack0,\hat{\tau}].
\]

\end{description}
\end{theorem}

\begin{proof}
The control range $I$ is compact and convex and the Lie algebra rank condition
holds, since $K\in\mathrm{PLARC}(A,B)$. Hence the projected control system in
$\mathbb{RP}^{d-1}$ satisfies the assumptions of \cite[Theorem 7.3.3]{ColKli}.
It follows that the control system $(\arabic{cs_2})_{I}$ has a unique
invariant control set $C$, which is compact, has nonempty interior, and is
contained in the closure of every attainable set of $(\arabic{cs_2})_{I}$.
Recall that an invariant control set is characterized by condition (i)
together with the property that every element of $C$ is approximately
controllable from every other element of $C$ (cf.~\cite[Definition
3.1.3]{ColKli}). The proof is completed by noticing that \cite[Lemma
3.2.21]{ColKli} implies assertion (ii) stating exact controllability to points
in the interior of $C$.
\end{proof}

We turn to growth rates for PE systems. Given $\alpha\in\mathcal{{G}}(T,\mu)$
and $0\not =x_{0}\in\mathbb{R}^{d}$, we say that $(\alpha,x_{0})$ is
$\#$-admissible for $A,B,K$ if there exists $\tau>0$ such that $\alpha$ is
$\tau$-periodic as well as
the trajectory $\Pi x(\cdot;x_{0},A,B,K,\alpha)$ in $\mathbb{RP}^{d-1}$
corresponding to $\alpha$ and starting at $q_{0}=\Pi x_{0}\in\mathbb{RP}%
^{d-1}$. Corresponding rates of convergence and divergence are defined by
replacing in the definitions of $\mathrm{rc},\mathrm{RC},\mathrm{rd}%
,\mathrm{RD}$ the class of trajectories $x(\cdot;0,x_{0},A,B,K,\alpha)$
corresponding to $(T,\mu)$-signals by the subclass corresponding to pairs
$(\alpha,x_{0})$ that are $\#$-admissible for $A,B,K$. More precisely, let%
\[
\mathrm{rc}_{\#}(A,B,K):=\inf-\lambda^{+}(x_{0},A,B,K,\alpha)\text{ and
}\mathrm{rd}_{\#}(A,B,K):=\inf\lambda^{-}(x_{0},A,B,K,\alpha),
\]
where in both cases the infimum is taken over all $(\alpha,x_{0})$ which are
$\#$-admissible for $A,B,K$. If the considered $A,B,K$ are clear from the
context, we omit these arguments here and in other expressions. Furthermore,
let
\[
\mathrm{RC}_{\#}(A,B):=\sup_{K\in M_{m,d}(\mathbb{R})}\mathrm{rc}%
_{\#}(A,B,K)\text{ and }\mathrm{RD}_{\#}(A,B):=\sup_{K\in M_{m,d}(\mathbb{R}%
)}\mathrm{rd}_{\#}(A,B,K).
\]

\begin{lemma}
\label{l-fun} Let $I:=[\mu/T,1]$ and $K$ in $\mathrm{PLARC}(A,B)$. Consider
the set $C$ from Theorem~\ref{uniform-time} for the projected system
$(\arabic{cs_2})_{I}$. Fix a point $
\Pi x_{-}\in\mathrm{int}C$, and let
$x_{0}:=e^{(T-\mu)(A+BK)}x_{-}$. Then for every $\varepsilon>0$ there exists
$\bar{\tau}>0$ such that for every $t>\bar{\tau}$ and every $\alpha
\in\mathcal{{G}}(T,\mu)$ there exists $\alpha_{\#}\in\mathcal{{G}}(T,\mu)$
with $(\alpha_{\#},x_{0})$ $\#$-admissible for $A,B,K$ satisfying
\begin{equation}
\left\vert \lambda^{+}(x_{0},A,B,K,\alpha_{\#})-\frac{1}{t}\log|x(t;0,x_{0}%
,A,B,K,\alpha)|\right\vert <\varepsilon. \label{est0}%
\end{equation}

\end{lemma}

\begin{proof}
First note that $\Pi x_{0}\in C$, since $\Pi x_{-}$ is in the invariant
control set $C$ and the control $u\equiv1$ has values in $I=[\mu/T,1]$.

For every $\alpha\in\mathcal{{G}}(T,\mu)$ and $t>0$ consider the signal
$\alpha_{\#}^{t}$ obtained through the following procedure: Let $\alpha
_{\#}^{t}(s)=\alpha(s)$ for $s\in\lbrack0,t]$ and $\alpha_{\#}^{t}(s)=1$ for
$s\in(t,t+T-\mu]$. By Theorem~\ref{uniform-time} there exist a time $\hat
{\tau}$ independent of $\alpha(\cdot)$ and $t$ and a control $v^{t}%
:[0,\tau^{(t)}]\rightarrow\lbrack\mu/T,1]$ with $\tau^{(t)}\leq\hat{\tau}$
such that $\Pi x(\tau^{(t)};0,y^{t},v^{t})=\Pi x_{-}$, where $y^{t}%
:=x(t+T-\mu;0,x_{0},\alpha_{\#}^{t})$. The definition of $\alpha_{\#}^{t}$ is
then concluded by taking%
\[
\alpha_{\#}^{t}(s)=\left\{
\begin{array}
[c]{ccc}%
v^{t}(s-(T-\mu)) & \text{for} & s\in(t+T-\mu,t+T-\mu+\tau^{(t)}]\\
1 & \text{for} & s\in(t+T-\mu+\tau^{(t)},t+2(T-\mu)+\tau^{(t)}],
\end{array}
\right.
\]
and extending $\alpha_{\#}^{t}$ periodically on $\mathbb{R}$ with period%
\[
T^{(t)}:=t+2(T-\mu)+\tau^{(t)}.
\]
Then $\Pi x(T^{(t)};x_{0},A,B,K,\alpha_{\#}^{t})=\Pi x_{0}$, hence periodicity
in projective space holds. By construction, $\alpha_{\#}^{t}\in\mathcal{G}%
(T,\mu)$ and $(\alpha^{t},x_{0})$ is $\#$-admissible for $A,B,K$. Periodicity
in projective space and homogeneity of the evolution imply
\begin{align}
\lambda^{+}\left(  x_{0},\alpha_{\#}^{t}\right)   &  =\lim_{k\rightarrow
\infty}\frac{1}{kT^{(t)}}\log|x(kT^{(t)};0,x_{0},\alpha_{\#}^{t}%
)|\label{cs_0}\\
&  =\frac{1}{T^{(t)}}\log|x(T^{(t)};0,x_{0},\alpha_{\#}^{t})|.\nonumber
\end{align}
Notice now that for every $t>0$
\[
x(T^{(t)};0,x_{0},\alpha_{\#}^{t})=R^{(t)}x(t;0,x_{0},\alpha_{\#}^{t}),
\]
where $R^{(t)}$ denotes the principal fundamental solution on the interval
$\left[  t,T^{(t)}\right]  $ corresponding to $\alpha_{\#}^{t}$, evaluated at
time $T^{(t)}$.

Since $T^{(t)}-t\leq2(T-\mu)+\hat{\tau}$, Gronwall's lemma immediately yields
the existence of $C_{0}>1$ independent of $t$ and $\alpha(\cdot)$ such that
$\left\Vert R^{(t)}\right\Vert ,\left\Vert \left(  R^{(t)}\right)
^{-1}\right\Vert \leq C_{0}$ for all $t>0$.

Then $\alpha_{\#}^{t}(s)=\alpha(s)$ for $s\in\lbrack0,t]$ implies%
\[
\left\vert \log\left\vert x(T^{(t)};0,x_{0},\alpha_{\#}^{t})\right\vert
-\log\left\vert x(t;0,x_{0},\alpha)\right\vert \right\vert <\log C_{0}.
\]
It follows that
\begin{align*}
\lefteqn{\left\vert \frac{1}{T^{(t)}}\log\left\vert x(T^{(t)};0,x_{0}%
,\alpha_{\#}^{t})\right\vert -\frac{1}{t}\log\left\vert x(t;0,x_{0}%
,\alpha)\right\vert \right\vert }\\
&  \leq\frac{1}{T^{(t)}}\left\vert \log\left\vert x(T^{(t)};0,x_{0}%
,\alpha_{\#}^{t})\right\vert -\log\left\vert x(t;0,x_{0},\alpha_{\#}%
^{t})\right\vert \right\vert +\left\vert \frac{1}{T^{(t)}}-\frac{1}%
{t}\right\vert \left\vert \log\left\vert x(t;0,x_{0},\alpha_{\#}%
^{t})\right\vert \right\vert \\
\lefteqn{<\frac{1}{T^{(t)}}\left[  \log C_{0}+(2(T-\mu)+\hat{\tau})\frac{1}%
{t}\left\vert \log\left\vert x(t;0,x_{0},\alpha_{\#}^{t})\right\vert
\right\vert \right]  .}%
\end{align*}
Since $A+\alpha BK$ is uniformly bounded for $\alpha\in\lbrack0,1]$,
Gronwall's lemma again shows that
\[
\frac{1}{t}\left\vert \log\left\vert x(t;0,x_{0},\alpha_{\#}^{t})\right\vert
\right\vert <C_{1}%
\]
for a constant $C_{1}>0$, uniformly with respect to $t\geq1$ and $\alpha
\in\mathcal{G}(T,\mu)$. It follows that for $t\geq1$%
\[
\left\vert \frac{1}{T^{(t)}}\log\left\vert x(T^{(t)};0,x_{0},\alpha_{\#}%
^{t})\right\vert -\frac{1}{t}\log\left\vert x(t;0,x_{0},\alpha_{\#}%
^{t})\right\vert \right\vert <\frac{1}{t}\left(  \log C_{0}+(2T+\hat{\tau
})C_{1}\right)  .
\]
Assertion \eqref{est0} then follows from \eqref{cs_0} by taking $t\geq
\bar{\tau}:=1+\varepsilon^{-1}\left(  \log C_{0}+(2T+\hat{\tau})C_{1}\right)
$.
\end{proof}

The periodic approximation provided by Lemma \ref{uniform-time} gives the
following approximation result for the rate of convergence.

\begin{proposition}
\label{29} Let
$(A,B)$ be in $P_{d,m}$.
For every $K\in\mathrm{PLARC}(A,B)$, we have $\mathrm{rc}(A,B,K)=\mathrm{rc}%
_{\#}(A,B,K)$.
\end{proposition}

\begin{proof}
For every $K$ the inequality
\begin{equation}
\mathrm{rc}(A,B,K)\leq\mathrm{rc}_{\#}(A,B,K) \label{triv-in}%
\end{equation}
is trivially satisfied. In order to prove the converse inequality, we fix
$K\in\mathrm{PLARC}(A,B)$, a constant $m\in\mathbb{R}$ such that
$\mathrm{rc}(A,B,K)<m$, and we show that $\mathrm{rc}_{\#}(A,B,K)<m$. By
definition,
there exist a $(T,\mu)$-signal $\alpha_{0}$ and a vector $x_{0}$ such that
\begin{equation}
\lambda^{+}(x_{0},A,B,K,\alpha_{0})=\limsup_{t\rightarrow+\infty}\frac{1}%
{t}\log|x(t;0,x_{0},A,B,K,\alpha_{0})|>-m. \label{155}%
\end{equation}
For a given function $\alpha_{0}$ the maximal Lyapunov exponent $\Lambda
^{+}(A,B,K,\alpha_{0})$ is attained on every basis of $\mathbb{R}^{d}$ (see
\cite[Chapter 2]{cesari}).
Since, moreover, $\mathrm{int}C$ is nonvoid, the set%
\[
\left\{  e^{(T-\mu)(A+BK)}x\left\vert {}\right.  \Pi x\in\mathrm{int}%
C\right\}
\]
contains a basis of $\mathbb{R}^{d}$. This implies that in \eqref{155} the
point $x_{0}$ can be chosen in this set. Now we can apply Lemma~\ref{l-fun}
with $\varepsilon=\frac{1}{2}(\lambda^{+}(x_{0},A,B,K,\alpha_{0})+m)$,
$\alpha=\alpha_{0}$, and $t$ large enough such that
\[
\frac{1}{t}\log\left\vert x(t;0,x_{0},\alpha_{0})\right\vert >-m+\varepsilon.
\]
Thus there is $\alpha_{\#}^{t}$ such that $(x_{0},\alpha_{\#}^{t})$ is
$\#$-admissible with
\[
\lambda^{+}\left(  x_{0},A,B,K,\alpha_{\#}^{t}\right)  >-m.
\]
It follows that $\mathrm{rc}_{\#}(A,B,K)<m$.
\end{proof}

Next we analyze the relations between convergence and divergence rates using
time reversal in PE systems. The time reversed system corresponding to a
non-autonomous control system of the type $\dot{x}=Ax+\alpha(t)Bu$ is $\dot
{x}=-Ax-\alpha(-t)Bu$. This justifies the notation $\alpha_{-}(t)=\alpha(-t)$
for every signal $\alpha$.
Moreover, it is clear that the two systems have the same accessibility
properties and, in particular, that $\mathrm{PLARC}(A,B)=\mathrm{PLARC}%
(-A,-B)$.

The rates of convergence and divergence for $\#$-admissible pairs satisfy the
following property under time reversal.

\begin{proposition}
\label{time-reversal} Let $(A,B)$ be in $P_{d,m}$ and $K\in M_{m,d}%
(\mathbb{R})$. Then $\mathrm{rd}_{\#}(-A,-B,K)=\mathrm{rc}_{\#}(A,B,K)$.
\end{proposition}

\begin{proof}
Note that $(\alpha,x_{0})$ is $\#$-admissible for $A,B,K$ if and only if
$(\alpha_{-},x_{0})$ is $\#$-admissible for $-A,-B,K$ and $-\lambda^{+}%
(x_{0},A,B,K,\alpha)=\lambda^{-}(x_{0},-A,-B,K,\alpha_{-}).$

Then, by taking the infimum with respect to all $\#$-admissible pairs
$(\alpha,x_{0})$ for $A,B,K$ one concludes
\begin{align*}
\mathrm{rc}_{\#}(A,B,K)  &  =\inf\{-\lambda^{+}(x_{0},A,B,K,\alpha)\left\vert
{}\right.  (\alpha,x_{0})\text{ }\#\text{-admissible for }A,B,K\}\\
&  =\inf\{\lambda^{-}(x_{0},-A,-B,K,\alpha_{-})\left\vert {}\right.
(\alpha,x_{0})\text{ }\#\text{-admissible for }-A,-B,K\}\\
&  =\mathrm{rd}_{\#}(-A,-B,K).
\end{align*}

\end{proof}

Next we provide a property dealing with the rate of divergence and
$\#$-admissible pairs. The proof is an adaptation of \cite[\S 3.4]{CK-article}.

\begin{proposition}
\label{211} Let $(A,B)$ be in $P_{d,m}$ and $K\in\mathrm{PLARC}(A,B)$. Then
$\mathrm{rd}_{\#}(A,B,K)=\mathrm{rd}(A,B,K)$.
\end{proposition}

\begin{proof}
The inequality $\mathrm{rd}(A,B,K)\leq\mathrm{rd}_{\#}(A,B,K)$ holds
trivially. In order to show the converse define for $t>0$
\[
S^{t}(A,B,K)=\{R(t,\alpha,A,B,K)\mid\alpha\in\mathcal{{G}}(T,\mu
),~\alpha~t\text{-periodic}\},
\]
where $R(\cdot,\alpha,A,B,K)$ denotes the solution of the Cauchy problem
\begin{equation}
\dot{R}=(A+\alpha BK)R,\quad R(0)=\mathrm{Id}. \label{R}%
\end{equation}
Thus $S^{t}(A,B,K)$ consists of fundamental solutions for $\dot{x}=A+\alpha
BK$ evaluated at time $t$. Note that not every function $\alpha$ satisfying
the persistent excitation condition (\ref{EP}) on $[0,t]$ may be extended to a
$(T,\mu)$-signal in $\mathcal{{G}}(T,\mu)$. Set
\[
\delta(A,B,K)=\limsup_{t\rightarrow+\infty}\frac{1}{t}\sup_{g\in S^{t}%
(A,B,K)}\log\Vert g\Vert,\qquad\delta^{\ast}(A,B,K)=\liminf_{t\rightarrow
+\infty}\frac{1}{t}\inf_{g\in S^{t}(A,B,K)}\log m(g),
\]
where $m(\cdot)$ is the \emph{conorm} corresponding to the vector norm
$|\cdot|$, defined by
\[
m(g)=\min\{|gx|\,\mid|x|=1\}.
\]
Let us prove that $S^{t}(-A,-B,K)=\{g^{-1}\mid g\in S^{t}(A,B,K)\}=\left(
S^{t}(A,B,K)\right)  ^{-1}$ for every $t>0$. This follows by time reversal:
Set $V(s)=R(s,\alpha,-A,-B,K)$ for $s\in\mathbb{R}$, with $\alpha$
$t$-periodic and in $\mathcal{{G}}(T,\mu)$. Then, $V(0)=\mathrm{Id}$ and
\[
\dot{V}=(-A-\alpha BK)V.
\]
Fix $s\geq0$ and let $G_{s}(t)=V(t)V(s)^{-1}$. Then, for $t\in\lbrack0,s]$,
$G_{s}(s-t)=R(t,\alpha_{s},A,B,K)$, where $\alpha_{s}(\tau):=\alpha(s-\tau)$,
$\tau\in\mathbb{R}$. Therefore, for every $s\geq0$, $V(s)^{-1}=R(s,\alpha
_{s},A,B,K)$. Hence, $V(t)=R(t,\alpha_{-},A,B,K)^{-1}$, since $\alpha
_{t}=\alpha_{-}$ by $t$-periodicity.
Hence
\[
\delta(-A,-B,K)=\limsup_{t\rightarrow+\infty}\frac{1}{t}\sup\log\Vert
g\Vert,\qquad\delta^{\ast}(-A,-B,K)=\liminf_{t\rightarrow+\infty}\frac{1}%
{t}\inf\log m(g),
\]
where the supremum and the infimum are taken over all $g\in S^{t}%
(-A,-B,K)=(S^{t}(A,B,K))^{-1}$. Using as in \cite[Lemma 3.1]{CK-article} the
relation $\Vert g\Vert=m(g^{-1})^{-1}$, we have
\[
\inf_{g\in S^{t}(A,B,K)}\log m(g)=-\sup_{g\in(S^{t}(A,B,K))^{-1}}\log\Vert
g\Vert.
\]
Dividing by $t$ and taking the $\liminf$ as $t$ goes to $+\infty$ one finds
\begin{equation}
\delta^{\ast}(A,B,K)=-\delta(-A,-B,K). \label{deltas}%
\end{equation}
Recall that by Proposition~\ref{time-reversal}
\begin{equation}
\mathrm{rd}_{\#}(A,B,K)=\mathrm{rc}_{\#}(-A,-B,K). \label{ddee}%
\end{equation}
Let us next prove that
\begin{equation}
\delta^{\ast}(A,B,K)\leq\mathrm{rd}(A,B,K). \label{tobere}%
\end{equation}
For $t>0$ define
\[
Q^{t}(A,B,K)=\{R(t,\alpha,A,B,K)\mid\alpha\in\mathcal{{G}}(T,\mu)\},
\]
and set
\[
\delta^{\ast}(t):=\inf_{g\in S^{t}(A,B,K)}\log m(g),\quad\delta_{1}^{\ast
}(t):=
\inf_{g\in Q^{t}(A,B,K)}\log m(g).
\]
We next prove that
\begin{equation}
\delta^{\ast}(A,B,K)=\liminf_{t\rightarrow+\infty}\delta^{\ast}(t)/t=\liminf
_{t\rightarrow+\infty}\delta_{1}^{\ast}(t)/t. \label{X}%
\end{equation}
The first equality is clear by definition, and we trivially have $\delta
^{\ast}(t)\geq\delta_{1}^{\ast}(t)$ for every $t>0$, since $S^{t}%
(A,B,L)\subset Q^{t}(A,B,L)$. Let us now show that there exist $C,\bar{t}>0$
such that $\delta^{\ast}(t)-\delta_{1}^{\ast}(t)\leq C$ for every $t\geq
\bar{t}$.

Abbreviate $R(t,\alpha):=R(t,\alpha,A,B,K)$ and pick any $g:=R(t,\alpha)\in
Q^{t}(A,B,K)$ corresponding to $\alpha\in\mathcal{{G}}(T,\mu)$ and $t\geq T$.
We modify $\alpha$ in order to get a periodic $(T,\mu)$-signal $\tilde{\alpha
}\in\mathcal{{G}}(T,\mu)$ as follows:
\[
\tilde{\alpha}(s)=\left\{
\begin{array}
[c]{ccc}%
\alpha(s) & \text{for} & s\in\lbrack0,t-T]\\
1 & \text{for} & s\in(t-T,t)
\end{array}
\right.
\]
and extend $\tilde{\alpha}$ $t$-periodically to $\mathbb{R}$. Notice that
$\tilde{g}:=R(t,\tilde{\alpha})\in S^{t}(A,B,K)$ and%
\[
g=R(T,\alpha(\cdot+t-T))R(t-T,\alpha),\qquad\tilde{g}=e^{T(A+BL)}%
R(t-T,\alpha).
\]
We conclude%
\[
\left\Vert g^{-1}\right\Vert =\left\Vert R(t-T,\alpha)^{-1}R(T,\alpha
(\cdot+t-T))^{-1}\right\Vert \leq\left\Vert e^{T(A+BL)}\right\Vert \left\Vert
\tilde{g}^{-1}\right\Vert \left\Vert R(T,\alpha(\cdot+t-T))^{-1}\right\Vert .
\]
Since $\left\Vert R(T,\alpha(\cdot+t-T))^{-1}\right\Vert $ is bounded,
uniformly with respect to $\alpha$ and $t\geq T$, there is a constant $C>0$
such that $\Vert g^{-1}\Vert\leq C_{0}\Vert\tilde{g}^{-1}\Vert$. Then clearly
\[
\log m(g)=-\log\Vert g^{-1}\Vert\geq-\log\Vert\tilde{g}^{-1}\Vert-C=\log
m(\tilde{g})-C\geq\delta_{1}^{\ast}(t)-C.
\]
Taking the infimum over $g\in Q^{t}(A,B,L)$ one gets equality (\ref{X}).

If $K\in\mathrm{PLARC}(A,B)$ we conclude the proof of the proposition
by showing that
\[
\mathrm{rc}_{\#}(-A,-B,K)\leq-\delta(-A,-B,K).
\]
Combining it with \eqref{deltas}, \eqref{ddee}, \eqref{tobere} this will
prove, as desired,%
\begin{equation}
\mathrm{rd}_{\#}(A,B,K)=\mathrm{rc}_{\#}(-A,-B,K)\leq-\delta(-A,-B,K)=\delta
^{\ast}(A,B,K)\leq\mathrm{rd}(A,B,K). \label{Y}%
\end{equation}
Then equalities hold here. We emphasize that the assumption that
$K\in\mathrm{PLARC}(A,B)$ is only used at this stage of the argument.

For the sake of notational simplicity, let us prove the equivalent inequality
$\mathrm{rc}_{\#}(A,B,K)\leq-\delta(A,B,K)$ (recall that $\mathrm{PLARC}%
(A,B)=\mathrm{PLARC}(-A,-B)$).
Since
$\mathrm{int}C$
is nonempty, there is a basis of $\mathbb{R}^{d}$, say $e_{1},\dots,e_{d}$
with $\Pi e_{1},...,\Pi e_{d}\in\mathrm{int}C$.
Then, clearly, for every $t>0$ and every $g\in S^{t}(A,B,K)$,
\[
\frac{1}{t}\log\Vert g\Vert\leq\frac{1}{t}\log\left(  \max_{i=1,\dots
,d}|ge_{i}|\right)  +\frac{1}{t}\log d.
\]
Then, for every $\varepsilon>0$, there exist
$i\in\{1,\dots,d\}$, $t>0$ arbitrarily large, and $\alpha\in\mathcal{{G}%
}(T,\mu)$ such that $\alpha$ is $t$-periodic and
\[
\frac{1}{t}\log|x(t;0,e_{i},A,B,K,\alpha)|>\delta(A,B,K)-\varepsilon.
\]
Applying Lemma~\ref{l-fun} with $x_{0}=e_{i}$ and $\alpha$, $t$, $\varepsilon$
as above, we obtain that there exists $\alpha_{\#}$ such that $(\alpha
_{\#},e_{i})$ is $\#$-admissible for $A,B,K$ and
\[
\lim_{s\rightarrow+\infty}\frac{1}{s}\log|x(s;0,e_{i},A,B,K,\alpha
_{\#})|>\delta(A,B,K)-2\varepsilon.
\]
We deduce that $\mathrm{rc}_{\#}(A,B,K)\leq-\delta(A,B,K)$, as required.
\end{proof}

\begin{cor}
\label{cooro} If $K$ belongs to $\mathrm{PLARC}(A,B)$ then%
\[
\mathrm{rc}(A,B,K)=\mathrm{rd}(-A,-B,K)=-{\delta(A,B,K)}.
\]

\end{cor}

\begin{proof}
Since $K\in\mathrm{PLARC}(A,B)$,
\begin{equation}
\mathrm{rc}(A,B,K)=\mathrm{rc}_{\#}(A,B,K)=\mathrm{rd}_{\#}%
(-A,-B,K)=\mathrm{rd}(-A,-B,K), \label{4eq}%
\end{equation}
where the first equality follows from Proposition \ref{29}, the second
equality from Proposition \ref{time-reversal}, and the last equality from
Proposition~\ref{211} applied to the pair $(-A,-B)$.
Since equalities hold in (\ref{Y}), we also have%
\[
\mathrm{rd}(-A,-B,K)=-\delta(A,B,K).
\]

\end{proof}

\begin{remark}
{Using the second equality in Corollary \ref{cooro} and the definition of
}$\delta$, one can also show that {$\mathrm{rc}(A,B,K)$ is equal to the
supremal Bohl exponent
\[
\sup_{\alpha\in\mathcal{G}(T,\mu)}\limsup_{s,t-s\rightarrow+\infty}\frac
{1}{t-s}\log\Vert R(t,s,\alpha,A,B,K)\Vert,
\]
where }$R(t,s,\alpha,A,B,K)$ is the solution of the differential equation in
(\ref{R}) with {initial condition $R(s)=$\textrm{Id}. This is analogous to
\cite[Theorem 7.2.20]{ColKli}.}
\end{remark}

\section{Continuity properties of the growth rates\label{Section4}}

We investigate in this section the continuity properties of the convergence
and divergence rates, i.e. of the maps $(A,B,K)\mapsto\mathrm{rc}(A,B,K)$ and
$(A,B,K)\mapsto\mathrm{rc}(A,B,K)$ defined in (\ref{td0}). This issue can be
restated as the study of the continuity properties of the maximal and minimal
Lyapunov exponents of system (\ref{system1}) with respect to $(A,B,K)$.

{
Denote by $\theta$ the flow
\[
\theta_{t}:\alpha\mapsto\alpha(t+\cdot),\qquad t\in\mathbb{R},
\]
defined on $\mathcal{G}(T,\mu)$. Clearly, the periodic points of the shift are
the periodic $(T,\mu)$-signals. }

\begin{lemma}
The periodic $(T,\mu)$-signals are dense in $\mathcal{{G}}(T,\mu)$.
\end{lemma}

\begin{proof}
Let $\alpha\in\mathcal{{G}}(T,\mu)$. We construct a sequence of periodic
$(T,\mu)$-signals $\alpha_{k}$ weak-$\star$ converging to $\alpha$.
Define
\[
\alpha_{k}(t)=\left\{
\begin{array}
[c]{ccc}%
1 & \text{for} & [-T+\mu-k,-k),\\
\alpha(t) & \text{for} & t\in\lbrack-k,k],\\
1 & \text{for} & t\in(k,,k+T-\mu],
\end{array}
\right.
\]
and extend on $\mathbb{R}$ by $2(k+T-\mu)$-periodicity. Then $\alpha_{k}$
belongs to $\mathcal{{G}}(T,\mu)$.
Take $y\in L^{1}(\mathbb{R},\mathbb{R})$ and let $\varepsilon>0$. There exists
$k_{\varepsilon}\in\mathbb{N}$ such that for all $k\geq k_{\varepsilon}$%
\[
\int_{\mathbb{R}\setminus\lbrack-k,k]}\left\vert y(t)\right\vert
dt<\varepsilon.
\]
Then, for $k>k_{\varepsilon}$,
\[
\left\vert \int_{\mathbb{R}}y(t)\alpha_{k}(t)dt-\int_{\mathbb{R}}%
y(t)\alpha(t)dt\right\vert \leq\int_{\mathbb{R}}\left\vert y(t)\right\vert
\left\vert \alpha_{k}(t)-\alpha(t)\right\vert dt\leq\int_{\mathbb{R}%
\setminus\lbrack-k,k]}\left\vert y(t)\right\vert dt<\varepsilon.
\]
\end{proof}

Since $\mathcal{G}(T,\mu)$ is compact connected metrizable for the
weak-$\star$ topology, the above lemma yields that the flow $\theta$ in the
base $\mathcal{{G}}(T,\mu)$ is chain transitive and the flows $\Phi
=\Phi(A,B,K)$ on $\mathcal{{G}}(T,\mu)\times\mathbb{R}^{d}$ depend
continuously on $(A,B,K)$. Thus the flow $\Phi$ satisfies the assumptions in
\cite[Corollary 5.3.11]{ColKli} and upper semicontinuity of the supremal
spectral growth rates follows. More precisely, for a sequence $(A^{n}%
,B^{n},K^{n})\rightarrow(A,B,K)$ with Lyapunov exponents denoted by
$\lambda^{n}$ one has%
\[
\sup\{\lambda\in\mathbb{R}\left\vert {}\right.  \ \text{there are Lyapunov
exponents }\lambda^{n}\text{ with }\lambda^{n}\rightarrow\lambda\}\leq
\sup_{\alpha\in\mathcal{{G}}(T,\mu),x_{0}\not =0}\lambda^{+}(x_{0}%
,A,B,K,\alpha).
\]
(Here, in addition to \cite[Corollary 5.3.11]{ColKli}, it is used that the
supremum of the Morse spectrum coincides with the supremum over all Lyapunov
exponents, \cite[Theorem 5.1.6]{ColKli}.) \ In the same way, one obtains that
\[
\inf\{\lambda\in\mathbb{R}\left\vert {}\right.  \ \text{there are Lyapunov
exponents }\lambda^{n}\text{ with }\lambda^{n}\rightarrow\lambda\}\geq
\inf_{\alpha\in\mathcal{{G}}(T,\mu),x_{0}\not =0}\lambda^{-}(x_{0},A^{0}%
,B^{0},K^{0},\alpha).
\]
An immediate consequence is that%
\[
\underset{n\rightarrow\infty}{\lim\sup}\sup_{\alpha\in\mathcal{{G}}%
(T,\mu),x_{0}\not =1}\lambda^{+}(x_{0},A^{n},B^{n},K^{n},\alpha)\leq
\sup_{\alpha\in\mathcal{{G}}(T,\mu),x_{0}\not =0}\lambda^{+}(x_{0},A^{0}%
,B^{0},K^{0},\alpha).
\]
In other words, the maximal Lyapunov exponent is upper semicontinuous or,
equivalently, the rate of convergence $\mathrm{rc}(A,B,K)$ is lower
semicontinuous with respect to $(A,B,K)$. Analogously,%
\[
\underset{n\rightarrow\infty}{\lim\inf}\inf_{\alpha\in\mathcal{{G}}%
(T,\mu),x_{0}\not =1}\lambda^{-}(x_{0},A^{n},B^{n},K^{n},\alpha)\geq
\inf_{\alpha\in\mathcal{{G}}(T,\mu),x_{0}\not =1}\lambda^{-}(x_{0},A^{0}%
,B^{0},K^{0},\alpha),
\]
Hence the minimal Lyapunov exponent is lower semicontinuous or, equivalently,
the rate of divergence $\mathrm{rd}(A,B,K)$ is lower semicontinuous with
respect to $(A,B,K)$.

\begin{theorem}
\label{lsc-usc} (i) The functions $\mathrm{rc},\mathrm{rd}:P_{d,m}\times
M_{m,d}(\mathbb{R})\rightarrow\mathbb{R},(A,B,K)\mapsto\mathrm{rc}(A,B,K)$ are
lower semicontinuous.
(ii) The restrictions of $\mathrm{rc}$ and $\mathrm{rd}$ to the set of all
$(A,B,K)$ with
$K\in\mathrm{PLARC}(A,B)$ are also upper semicontinuous, and hence
continuous there.
\end{theorem}

\begin{proof}
Assertion (i) has been established above. We show upper semicontinuity of
$\mathrm{rc}(A,B,K)$, i.e., lower semicontinuity of the maximal Lyapunov
exponent on $\{(A,B,K)\mid K\in\mathrm{PLARC}(A,B)\}$.

Consider a sequence $(A^{n},B^{n},K^{n})\rightarrow(A^{0},B^{0},K^{0})$. We
have to show that%
\begin{equation}
\underset{n\rightarrow\infty}{\lim\sup~}\mathrm{rc}(A^{n},B^{n},K^{n}%
)\leq\mathrm{rc}(A^{0},B^{0},K^{0}), \label{ass}%
\end{equation}
that is%
\[
\underset{n\rightarrow\infty}{\lim\inf}\sup_{\alpha\in\mathcal{{G}}%
(T,\mu),x_{0}\not =0}\lambda^{+}(x_{0},A^{n},B^{n},K^{n},\alpha)\geq
\sup_{\alpha\in\mathcal{{G}}(T,\mu),x_{0}\not =0}\lambda^{+}(x_{0},A^{0}%
,B^{0},K^{0},\alpha).
\]
Let $\varepsilon>0$. Since $\mathrm{rc}(A^{0},B^{0},K^{0})=\mathrm{rc}%
_{\#}(A^{0},B^{0},K^{0})$, there exists a $\#$-admissible pair $(\alpha
_{\varepsilon},x_{0,\varepsilon})$ such that
\[
\left\vert \mathrm{rc}(A^{0},B^{0},K^{0})-\frac{1}{\tau_{\varepsilon}}%
\log\left\vert \mu_{\varepsilon}\right\vert \right\vert <\varepsilon,
\]
where $\tau_{\varepsilon}$ is the period of the trajectory on $\mathbb{RP}%
^{d-1}$ associated with $\alpha_{\varepsilon}$ and starting from $\Pi
x_{0,\varepsilon}$ and $\mu_{\varepsilon}\in\mathbb{R}$ satisfies
$R_{\varepsilon}(\tau_{\varepsilon})x_{0,\varepsilon}=\mu_{\varepsilon
}x_{0,\varepsilon}$, with $R_{\varepsilon}(\cdot)$ the fundamental matrix%
\[
\dot{R}_{\varepsilon}(t)=(A^{0}+\alpha_{\varepsilon}(t)B^{0}K^{0}%
)R_{\varepsilon}(t),\quad R_{\varepsilon}(0)=\mathrm{Id}_{d}.
\]

Furthermore, recall that eigenvalues depend continuously on the matrix (this
follows, e.g., from \cite[Lemma A.4..1]{Sontag}). For $n\in\mathbb{N}$, let
$R^{n}$ be the fundamental matrix%
\[
\dot{R}^{n}(t)=(A^{n}+\alpha_{\varepsilon}(t)B^{n}K^{n})R^{n}(t),\quad
R^{n}(0)=\mathrm{Id}_{d}.
\]
Then there exists a sequence $(\mu_{\varepsilon}^{n})_{n\in\mathbb{N}}$ in
$\mathbb{C}$ converging to $\mu_{\varepsilon}$ as $n$ tends to infinity such
that $\mu_{\varepsilon}^{n}$ is an eigenvalue of $R^{n}(\tau_{\varepsilon})$
for $n\in\mathbb{N}$. One therefore has for $n$ large enough that
\[
\left\vert \frac{1}{\tau_{\varepsilon}}\log\left\vert \mu_{\varepsilon
}\right\vert -\frac{1}{\tau_{\varepsilon}}\log\left\vert \mu_{\varepsilon}%
^{n}\right\vert \right\vert <\varepsilon.
\]
Furthermore, there is $x_{\varepsilon}^{n}\in\mathbb{S}^{d-1}$ in the
generalized real eigenspace associated with $(\mu_{\varepsilon}^{n}%
,\overline{\mu_{\varepsilon}^{n}})$ such that%
\[
|R^{n}(k\tau_{\varepsilon})x_{\varepsilon}^{n}|=|\mu_{\varepsilon}^{n}%
|^{k},\qquad\mbox{for every }k\in\mathbb{N}.
\]
This implies
\[
\sup_{\alpha\in\mathcal{{G}}(T,\mu),x_{0}\not =0}\lambda^{+}(x_{0},A^{n}%
,B^{n},K^{n},\alpha)\geq\lambda^{+}(x_{\varepsilon}^{n},A^{n},B^{n}%
,K^{n},\alpha_{\varepsilon})=\frac{1}{\tau_{\varepsilon}}\log\left\vert
\mu_{\varepsilon}^{n}\right\vert ,
\]
and hence for $n$ large enough%
\[
\mathrm{rc}(A^{n},B^{n},K^{n})\leq\frac{1}{\tau_{\varepsilon}}\log\left\vert
\mu_{\varepsilon}^{n}\right\vert \leq\frac{1}{\tau_{\varepsilon}}%
\log\left\vert \mu_{\varepsilon}\right\vert +\varepsilon\leq\mathrm{rc}%
(A^{0},B^{0},K^{0})+2\varepsilon.
\]
Since $\varepsilon$ is arbitrary, assertion (\ref{ass}) follows.

Finally, upper semicontinuity of $\mathrm{rd}(A,B,K)$ is a consequence of
Corollary\textbf{ }\ref{cooro}.
\end{proof}

\section{Properties of maximal growth rates}

\label{s-5}
Before stating the main result of the paper, we need the following proposition.


\begin{prop}\label{last1}
\label{ptrivial}
Let $(A,B)\in P_{d,m}.$ 
The following statements are equivalent:
\begin{description}
\item[$(i)$] $\mathrm{PLARC}(A,B)$ is nonempty;
\item[$(ii)$] $\mathrm{LARC}_0(A,B)$ is nonempty;
\item[$(iii)$] $\mathrm{LARC}_0(A,B)$  is dense in $M_{m,d}(\mathbb{R})$;
\item[$(iv)$] $\mathrm{PLARC}(A,B)$ is dense in $M_{m,d}(\mathbb{R})$.
 \end{description}
Moreover, 
if $d\geq 3$, then the above statements are also equivalent to the following ones
\item[$(v)$] $\mathrm{LARC}(A,B)$ is nonempty;
\item[$(vi)$] $\mathrm{LARC}(A,B)$ is dense in $M_{m,d}(\mathbb{R})$.
\end{prop}
\begin{proof}
Notice that the implication $(iii)\Rightarrow (iv)$ follows directly from Lemma~\ref{proj}, while $(iv)\Rightarrow (i)$
 is trivial.

We next prove that $(i)$ implies $(ii)$. Let $K^\star$ be in $\mathrm{PLARC}(A,B)$. Since $\RP^{d-1}$ is compact, there exists an open neighborhood $V_0$ of $K^\star$ contained in $\mathrm{PLARC}(A,B)$. 
For every $K\in V_0$, let $G_K^0$ and $L_K^0$ be, respectively, the group and the Lie algebra generated by $A-(\Tr(A)/d)\Id_d$ and $BK-(\Tr(BK)/d)\Id_d$. 
As noticed in Remark~\ref{irre}, $G_K^0$ acts transitively on $\RP^{d-1}$, which implies that $L^0_K$ belongs, up to similarity, to the list given \cite[Theorem 19]{dirr-helmke} (based on results given in \cite{volklein}).
To close the argument, it is enough to find $K\in V_0$ such that no matrix similar to $BK-(\Tr(BK)/d)\Id_d$ belongs to $\mbox{so}(d)$ nor $\mbox{spin}(9,1)$. 
Notice that the spectrum of any matrix similar to an element of $\mbox{so}(d)$ is contained in the imaginary 
axis, while the eigenvalues of a matrix similar to an element of $\mbox{spin}(9,1)$ are symmetric with respect to the origin, in the sense that if 
$\lambda$ is an eigenvalue then $-\lambda$ is as well, with the same algebraic multiplicity as $\lambda$ (a proof of this fact is provided in Appendix). It is therefore enough to find $K\in V_0$ and $\lambda$ nonzero so that 
$\lambda$ is an eigenvalue of
$BK-(\Tr(BK)/d)\Id_d$ and either $-\lambda$ is not or its multiplicity is different from the one of $\lambda$.

With no loss of generality, we assume that $B=\begin{pmatrix} \Id_m\\ 0\end{pmatrix}$ and thus 
$$F(K):=BK-(\Tr(BK)/d)\Id_d=\begin{pmatrix} K_1-(\Tr(K_1)/d)\Id_m&K_2\\ 0& -(\Tr(K_1)/d)\Id_{d-m}\end{pmatrix},$$ where $K=\begin{pmatrix} K_1& K_2\end{pmatrix}$ is an arbitrary matrix of $M_{m,d}(\mathbb{R})$.

If $m=d$, then clearly $\{F(K)\mid K\in V_0\}$ is an open nonempty subset of $\mbox{sl}(d,\R)$
and one concludes.

If $m=d-1$, set $K_t=K^\star+t\diag(\Id_{d-1},-(d-1))$ for $t\in \R$. Clearly $K_t\in V_0$ for $t$ small enough. 
For $t\ne 0$ small enough, 
$-(\Tr(K^\star_1)/d)-(d-1)t$ is a nonzero real eigenvalue of $F(K_t)$ and the other eigenvalues are of the type $\lambda-\Tr(K^\star_1)/d+t/d$ with $\lambda$ eigenvalue of $K_1^\star$. 
Then, if  
$\Tr(K^\star_1)/d+(d-1)t$ is an eigenvalue of $F(K_t)$, 
then there exists $\lambda$ eigenvalue of $K_1^\star$ so that
$$ \lambda-2 \frac{\Tr(K^\star_1)}d=\left(d-1-\frac 1d\right)t.$$
Hence, for $t\ne 0$ small enough, 
$\Tr(K^\star_1)/d+(d-1)t$ cannot be an eigenvalue of $F(K_t)$ and 
we are done. 

%
%
%
%
%
%

Assume now that $m\leq d-2$. Then there exists $K= \begin{pmatrix} K_1& K_2\end{pmatrix}\in V_0$ such that $k_1:= -\Tr(K_1)/d\ne 0$ and
the eigenvalues of $K_1-(\Tr(K_1)/d)\Id_m$ and $k_1$ are two by two distinct. As a consequence, $k_1$ is a nonzero eigenvalue of $F(K)$ of multiplicity at least two and the multiplicity of $-k_1$ as  an eigenvalue of $F(K)$ 
is at most one. As noticed above, this allows to conclude that any matrix similar to $F(K)$ does not belong to $\mbox{so}(d)$ if $d\ne 10$ and to the union of $\mbox{so}(10)$  and $\mbox{spin}(9,1)$ if $d=10$.

We conclude the proof of the first part of the statement 
 by showing that $(ii)$ implies $(iii)$.
Assume that there exists $K_{0}\in
M_{m,d}(\mathbb{R})$ such that $\mathrm{Lie}(A-(\Tr(A)/d)\Id_d,B K_{0}-(\Tr(B K_{0})/d)\Id_d)=\mbox{sl}(d,\mathbb{R})$.
Let us select a basis of $\mbox{sl}(d,\mathbb{R})$ made of iterated Lie brackets of
$A-(\Tr(A)/d)\Id_d$ and $B K_{0}-(\Tr(B K_{0})/d)\Id_d$, denoted by $Q_{1}(K_{0}),\dots,Q_{d^{2}-1}(K_{0})$.

For every $K\in M_{m,d}(\mathbb{R})$ and $j=1,\dots,d^{2}-1$, denote by
$Q_{j}(K)$ the iterated Lie bracket of $A-(\Tr(A)/d)\Id_d$ and $B K-(\Tr(B K)/d)\Id_d$ obtained by replacing
$K_{0}$ by $K$ in the Lie bracket expression of $Q_{j}(K_{0})$.

Consider each $Q_{j}(K)$ as a row vector of $\mathbb{R}^{d^{2}-1}$ (using, for instance, the representation on the basis $Q_{1}(K_{0}),\dots,Q_{d^{2}-1}(K_{0})$) and define
$f(K)=\det(Q_{1}(K),\dots,Q_{d^{2}-1}(K))$. Since $f$ is an analytic function of
the entries of $K$ and $f(K_{0})\ne0$, we deduce that $f$ cannot vanish on any
nonempty open subset of $M_{m,d}(\mathbb{R})$.
This proves the density of $\mathrm{LARC}_0(A,B)$.

Let us now prove that $(iv)$ implies $(vi)$ when $d\geq 3$. 
The proof of the proposition is then concluded by noticing that $(vi)$ trivially implies $(v)$, which implies $(i)$ by Lemma~\ref{proj}.
Let  $K\in \LARC_0(A,B)$. Notice in particular that $B\ne 0$. 
By a perturbative argument, one gets that
$$
\mathrm{Lie}(A-(\Tr(A)/d)\Id_d,B\tilde K-(\Tr(B\tilde K)/d)\Id_d)=\mbox{sl}(d,\R),
$$
for every $\tilde K$ in a neighborhood of $K$, 
and thus we may assume that $\Tr(BK)\ne 0$. 
We now show that 
$L_K=\mathrm{Lie}(A,BK)=M_n(\R)$. Notice that the map $\Phi:M\mapsto M-(\Tr(M)/d)\Id_d$ is surjective from $L_K$ to $\mbox{sl}(d,\R)$. Indeed, by hypothesis $\mbox{sl}(d,\R)=\Lie(\Phi(A),\Phi(BK))$ and $\Phi:L_K\to \mbox{sl}(d,\R)$ is a Lie algebra homomorphism. Hence, $L_K$ has codimension at most $1$ in $M_d(\R)$.
Using the fact that $d\geq 3$ and a theorem of Amayo (see \cite{amayo} and also \cite[Theorem 1.1]{kissin}), one gets that the only Lie subalgebra of $M_d(\R)$ of codimension $1$ is $\mbox{sl}(d,\R)$. 
Since $\Tr(BK)\ne0$ then $L_K\ne \mbox{sl}(d,\R)$ and therefore $L_K$ must be equal to $M_d(\R)$.
\end{proof}

%

\begin{remark}
The hypothesis $d\geq 3$ is essential in the proof of the implication $(iv)\Rightarrow (vi)$, where the fact that $\mbox{sl}(d,\R)$ is simple for $d\geq 3$ is crucial. 
Indeed, the latter is not true when $d=2$ and the  implication $(iv)\Rightarrow (vi)$ is not true, as illustrated by Proposition~\ref{Proposition_cont} below.
\end{remark}
In the context of control theory, it seems reasonable to address the issue of
a possible relationship between nonemptiness of $\mathrm{LARC}(A,B)$ (or
$\mathrm{PLARC}(A,B)$) and controllability of the pair $(A,B)$.

More precisely, is it true that if $\mathrm{PLARC}(A,B)$ is nonempty then
$(A,B)$ is controllable? The next proposition shows that the answer is yes,
except in trivial situations.

\begin{proposition}
\label{Proposition_cont}Let $(A,B)\in P_{d,m}$ be a non-controllable pair such
that $\mathrm{PLARC}(A,B)$ is nonempty. Then $B=0$,
$d=2$ and the eigenvalues of $A$ are non-real.
\end{proposition}

\begin{proof}
Consider a controllability decomposition of the pair $(A,B)$ of the form
\[
A=%
\begin{pmatrix}
A_{1} & A_{2}\\
0 & A_{3}%
\end{pmatrix}
,\quad B=%
\begin{pmatrix}
B_{1}\\
0
\end{pmatrix}
,
\]
where $(A_{1},B_{1})\in P_{r,m}$ is controllable with controllability index
$r<d$ and $A_{3}\in M_{d-r}(\mathbb{R})$.

Let $K\in\mathrm{PLARC}(A,B)$. It is clear that every matrix $C$ in the Lie
algebra $\mathrm{Lie}(A,BK)$ is of the form
\[
C=%
\begin{pmatrix}
C_{1} & C_{2}\\
0 & C_{3}%
\end{pmatrix}
,
\]
with $C_{1}\in M_{r}(\mathbb{R})$, $C_{3}\in M_{d-r}(\mathbb{R})$.

If $r>0$ then the Lie algebra $\mathrm{Lie}(A,BK)$ is reducible, contradicting
$K\in\mathrm{PLARC}(A,B)$, as noticed in Remark~\ref{irre}. This proves that
$r=0$, hence $B=0$.

Then the tangent space to any orbit of $\dot q=(\Pi A) q$, $q\in
\mathbb{RP}^{d-1}$, is of dimension at most one, implying that $d=2$. 
Then $A$ does not have an
invariant space of dimension one, hence the result.

\end{proof}

{\color{blue} }


%

Thanks to the results of the previous sections, we are ready to state the main
result of the paper, equality between maximal growth rates for PE systems.

\begin{theorem}
\label{5p1} Let $(A,B)\in P_{d,m}$ and assume that $\mathrm{LARC_0}(A,B)$ is
{
nonempty}.
Then for persistently excited systems of the form (\ref{system1})
the maximal rate of convergence defined in (\ref{tdd}) and the maximal rate of
divergence defined in (\ref{tDD}) satisfy%
\[
\mathrm{RC}(A,B)
=\mathrm{RD}(-A,-B).
\]

\end{theorem}

\begin{proof}
According to Proposition \ref{ptrivial}, one can assume that $\mathrm{PLARC}(A,B)$  is non empty. By definition of $\mathrm{RC}(A,B)$, for every $\xi<\mathrm{RC}(A,B)$ there
exists $K_{\xi}\in M_{d,m}(\mathbb{R})$ such that $\mathrm{rc}(A,B,K_{\xi
})>\xi$. 
According to Proposition \ref{last1} given below, the nonemptiness of $\mathrm{PLARC}(A,B)$ is equivalent to its density in $M_{m,d}(\mathbb{R})$. 
As a consequence, there
exists a sequence $(K_{\xi}^{n})_{n\in\mathbb{N}}$ in $\mathrm{PLARC}(A,B)$
converging to $K_{\xi}$. Theorem~\ref{lsc-usc}(i)) shows lower semicontinuity
of $\mathrm{rc}(A,B,\cdot)$ on $M_{d,m}(\mathbb{R})$ and hence $\mathrm{rc}%
(A,B,K_{\xi}^{n})>\xi$ for $n$ large enough. Hence
\[
\mathrm{RC}(A,B)=\sup_{K\in\mathrm{PLARC}(A,B)}\mathrm{rc}(A,B,K)=\sup
_{K\in\mathrm{PLARC}(A,B)}\mathrm{rc}_{\#}(A,B,K),
\]
where the last equality follows from Proposition~\ref{29}.
Proposition~\ref{time-reversal} then implies that $\mathrm{RC}(A,B)=\sup
_{K\in\mathrm{PLARC}(A,B)}\mathrm{rd}_{\#}(-A,-B,K)$

Now by the lower semicontinuity of $\mathrm{rd}(A,B,\cdot)$ on $M_{m,d}%
(\mathbb{R})$ we obtain as claimed%
\begin{align*}
\sup_{K\in\mathrm{PLARC}(A,B)}\mathrm{rd}_{\#}(-A,-B,K)  &  =\sup
_{K\in\mathrm{PLARC}(A,B)}\mathrm{rd}(-A,-B,K)\\
&  =\sup_{K\in M_{m,d}(\mathbb{R})}\mathrm{rd}(-A,-B,K)=\mathrm{RD}%
(-A,-B)\text{.}%
\end{align*}

\end{proof}

\begin{remark}
According to Proposition \ref{ptrivial}, if one assumes that $d\geq 3$, then the nonemptiness hypothesis  on $\LARC_0(A,B)$ can be weakened to the nonemptiness of $\LARC(A,B)$.
\end{remark}

\section{The single-input case\label{Section5}}

In this section we assume $m=1$ and we use $b$ to denote the $d\times1$ matrix
$B$.
Let, moreover, $(e_{1},\dots,e_{d})$ be the canonical basis of $\mathbb{R}%
^{d}$.

\subsection{Conditions for $K$ to belong to $\mathrm{PLARC}(A,b)$}

Given a controllable pair $(A,b)\in P_{d,1}$, let $v(A,b)$ and $P(A,b)$ be,
respectively, the unique row vector in $M_{1, d}(\mathbb{R})$ and the unique
invertible matrix so that $(J_{d}+e_{d} v(A,b),e_{d})$ is the controllability
form of $(A-\mathrm{Tr}(A)\mathrm{Id},b)$ (cf. \cite{Sontag}) and $P(A,b)$ is
the corresponding change of coordinates, i.e.,
\[
P(A,b)^{-1}(J_{n}+e_{d} v(A,b))P(A,b)=A-\mathrm{Tr}(A)\mathrm{Id},\qquad
P(A,b)^{-1}e_{d}=b.
\]
Note that $v(A,b) e_{d}=0$ by construction.

We now provide the main result of the section, which ensures that
$\mathrm{PLARC}(A,b)$ is dense if the pair $(A,b)$ is controllable.

\begin{thm}
\label{th-rc} \label{corri} Let $(A,b)\in P_{d,1}$ be a controllable pair.
There exists $c>0$ such that, for every $K\in M_{1,d}(\mathbb{R})$ for which
the eigenvalues of $A+bK$ have either all real part smaller than $-c$ or all
real part larger than $c$, it follows that $
K\in\mathrm{LARC}(A-\mathrm{Tr}(A)\mathrm{Id},b)$. In particular,
$\mathrm{PLARC}(A,b)$ is dense in $M_{1,d}(\mathbb{R})$.
\end{thm}

\begin{remark}
Theorem~\ref{th-rc} does not generalize to the multi-input case under the sole
condition that $(A,B)$ is controllable. Think for instance of the system
$(A,B)=(0_{d},\mathrm{Id}_{d})$. Then clearly $\mathrm{RC}(A,B)=+\infty$ and
$\mathrm{PLARC}(A,B)=\emptyset$.
\end{remark}

The proof of Theorem~\ref{th-rc} is based on the following two technical results.

\begin{proposition}
\label{acc}
Let $(A,b)\in P_{d,1}$
be a controllable pair and $v(A,b)$ and $P(A,b)$ be defined as above. Fix
$K\in M_{1, d}(\mathbb{R})$. For $j\geq0$, set $K_{j}=K(J_{d}+e_{d}
v(A,b))^{j} $, and $r_{j}=K (J_{d}+e_{d} v(A,b))^{j} e_{d}$. Then $KP(A,b)
\in\mathrm{LARC}(A-\mathrm{Tr}(A)\mathrm{Id},b)$ if $r_{j}\ne0$ for every
$j=0,\dots,d-1$ and $K_{0},\dots,K_{d-1}$ are linearly independent. In
particular,
$\mathrm{LARC}(J_{d},e_{d})$ contains all line vectors with coefficients in
$\mathbb{R}\setminus\{0\}$.
\end{proposition}

\begin{proof}
In this proof we abbreviate $v:=v(A,b)$ and $P:=(A,b)$. Since
\begin{equation}
\mathrm{LARC}(J_{d}+e_{d}v,e_{d})=\mathrm{LARC}(P(A-\mathrm{Tr}(A)\mathrm{Id}%
)P^{-1},Pb)=\mathrm{LARC}(A-\mathrm{Tr}(A)\mathrm{Id},b)P^{-1} \label{larcs}%
\end{equation}
we can assume without loss of generality that $(A,b)$ is in controllable form
and that $\mathrm{Tr}(A)=0$, i.e., $A=J_{d}+e_{d}v$, $b=e_{d}$, $v\,e_{d}=0$.

For $j\geq0$, define
\[
f_{d-j}=(J_{d}+e_{d} v)^{j} e_{d}.
\]

We next prove that the rank-1 matrices $f_{j}K_{l}$, $j=1,\dots,d$,
$l=0,\dots,d-1$, belong to
\[
\mathcal{L}:=\mathrm{Lie}(J_{d}+e_{d}v,e_{d}K).
\]

Notice that for $j,l\geq0$
\[
K_{l} f_{d-j}=r_{l+j}.
\]

Straightforward computations yield for $j\leq d$ and $l\geq0$
\begin{align*}
[J_{d}+e_{d} v,f_{j} K_{l}]  &  =f_{j} K_{l+1} -f_{j-1}K_{l},\\
[f_{j} K_{l},[J_{d}+e_{d} v,f_{j} K_{l}]]  &  =-2 r_{d+l+1-j} f_{j}
K_{l}-r_{d+l-j}(f_{j-1} K_{l}+f_{j} K_{l+1}).
\end{align*}

Hence, if $f_{j} K_{l}$ is in $\mathcal{L}$ and $r_{d+l-j}$ is different from
zero, then $f_{j} K_{l+1}$ and $f_{j-1}K_{l}$ also belong to $\mathcal{L}$.

By a trivial induction one deduces that $f_{j} K_{l}\in\mathcal{L}$ for $j\leq
d$, $l\geq0$ and $1\leq j-l\leq d$, since $e_{d} K=f_{d} K_{0}\in\mathcal{L}$
and $r_{0},\dots,r_{d-1}$ are not zero.

We prove by induction on $m=l-j$ that the following property holds true:

($\mathcal{P}_{m}$) For every $j\leq d$ and $l\geq0$ such that $j-l\geq m$ we
have $f_{j}K_{l}\in\mathcal{L}$.

We proved $\mathcal{P}_{m}$ for $m$ up to $-1$. Assume that $P_{m}$ holds true
for $m\geq-1$. Notice that for $j\leq d-1$ and $l\leq d-2$ one has
\[
\lbrack f_{j}K_{l},f_{j+1}K_{l+1}]=r_{d-1-(j-l)}f_{j}K_{l+1}-r_{d+1-(j-l)}%
f_{j+1}K_{l}%
\]
and $r_{d-1-(j-l)}\neq0$. If $l-j=m$ then $f_{j+1}K_{l}\in\mathcal{L}$ and we
conclude that also $f_{j}K_{l+1}$ is in $\mathcal{L}$, i.e., $\mathcal{P}%
_{m+1}$ holds true.

We have proved, as claimed, that $f_{j} K_{l}\in\mathcal{L}$ for $j=1,\dots,d$
and $l=0,\dots,d-1$. Since $f_{1},\dots, f_{d}$ and $K_{0},\dots,K_{d-1}$ are
linearly independent, respectively by construction and by hypothesis, then it
follows that $\mathcal{L}=M_{d}(\mathbb{R})$, concluding the proof of the
proposition.
\end{proof}

\begin{proposition}
\label{25}
If all the real parts of the eigenvalues of $J_{d}+e_{d}K$ are nonzero and
have the same sign, then
$|k_{d-m}|\geq c_{0}|k_{d-m+1}|(d-m)/(m+1)$ for every $m=0,\dots,d-1$, with
$k_{d+1}:=1$ and $c_{0}:=\min_{\lambda\in\sigma(J_{d}+e_{d}K)}|\Re(\lambda)|$.
\end{proposition}

\begin{proof}
Denote by $\lambda_{1},\dots,\lambda_{d}$ the eigenvalues of $J_{d}+e_{d}K$.
Notice that $|k_{d}|=|\lambda_{1}+\cdots+\lambda_{d}|=|\Re(\lambda
_{1})|+\cdots+|\Re(\lambda_{d})|\geq d\,c_{0}$.

For $m=1,\dots,d$,
\[
|k_{d-m+1}|=\left|  \sum_{j_{1}<\cdots<j_{m}} \lambda_{j_{1}}\cdots
\lambda_{j_{m}}\right|  =\sum_{j_{1}<\cdots<j_{m}} |\Re(\lambda_{j_{1}}%
\cdots\lambda_{j_{m}})|.
\]

Since
\[
\sum_{j_{1}<\cdots<j_{m}}\left(  \lambda_{j_{1}}\cdots\lambda_{j_{m}}%
\sum_{j\neq j_{1},\dots,j_{m}}\lambda_{j}\right)  =(m+1)\sum_{j_{1}%
<\cdots<j_{m+1}}\lambda_{j_{1}}\cdots\lambda_{j_{m+1}},
\]
we have
\begin{align*}
(m+1)|k_{d-m}|  &  =\left\vert \sum_{j_{1}<\cdots<j_{m}}\left(  \lambda
_{j_{1}}\cdots\lambda_{j_{m}}\sum_{j\neq j_{1},\dots,j_{m}}\lambda_{j}\right)
\right\vert \\
&  \geq\sum_{j_{1}<\cdots<j_{m}}\sum_{j\neq j_{1},\dots,j_{m}}|\Re
(\lambda_{j_{1}}\cdots\lambda_{j_{m}})||\Re(\lambda_{j})|\\
&  \geq\sum_{j_{1}<\cdots<j_{m}}(n-m)|\Re(\lambda_{j_{1}}\cdots\lambda_{j_{m}%
})|c_{0}=(d-m)c_{0}|k_{d-m+1}|.
\end{align*}
This concludes the proof.
\end{proof}

We can now prove Theorem~\ref{th-rc}.

\medskip

\noindent\emph{Proof of Theorem~\ref{th-rc}.} According to \eqref{larcs}, one
must prove the assertion for $(A,b)=(J_{d}+e_{d}v,e_{d})$ with $ve_{d}=0$.

Let then $K\in M_{1,d}(\mathbb{R})$ be such that the eigenvalues of
$J_{d}+e_{d}(v+K)$ have either all real part smaller than $-c$ or all real
part larger than $c$
for some $c>0$ to be chosen later. Applying Proposition~\ref{25} one gets that
for every $m=1,\dots,d-1$,
\[
|k_{d-m}+v_{d-m}|\geq c|k_{d-m+1}+v_{d-m+1}|(d-m)/(m+1),
\]
and $|k_{d}|\geq cd$, where $v=(v_{1},\dots,v_{d-1},0)$. By taking $c$ large
enough with respect to $|v|$ we have
\begin{equation}
|k_{d-m}|\geq\frac{c}{2d}|k_{d-m+1}|,\qquad m=1,\dots,d-1. \label{C2}%
\end{equation}
In particular, $k_{m}\neq0$ for $m=1,\dots,d$. According to
Proposition~\ref{acc} the proof is completed if we show that $r_{m}%
=K(J_{d}+e_{d} v)^{m} e_{d}\neq0$, $m=0,\dots,d-1$, and $K_{0},\dots,K_{d-1}$
are linearly independent.

By definition of $r_{m}$ a trivial induction yields
\[
r_{m}=k_{d-m}+\sum_{j<d-m}\alpha_{j}^{m}k_{j}%
\]
for some constants $\alpha_{j}^{m}$ independent of $K$. Fix $m\in
\{0,\dots,d-m\}$. If all the corresponding $\alpha_{j}^{m}$ are equal to zero
then we are done, otherwise let $\bar{\jmath}$ be the smallest index $j$ so
that $\alpha_{j}^{m}\neq0$. Then
\[
r_{m}=\alpha_{\bar{\jmath}}^{m}k_{\bar{\jmath}}(1+\xi)
\]
with $|\xi|=O(1/c)$ according to \eqref{C2}. The conclusion follows from $c$
large enough.

Similarly, one has that $K=k_{1}(e_{1}^{T}+\Xi)$ with $|\Xi|=O(1/c)$. One
immediately deduces that the for $m=0,\dots,d-1$, $K_{m}=k_{1}(e_{m+1}%
^{T}+O(1/c))$. Hence, $K_{0},\dots,K_{d-1}$ are linearly independent for $c$
large enough.

The last part of the statement follows from Lemma~\ref{proj} and
Proposition~\ref{ptrivial}. \hfill$\blacksquare$

\subsection{Maximal growth rates in the single-input case}

All controllable single-input systems have the same controllability form.
Hence $\mathrm{RC}(A,b)=\mathrm{RC}(A,b^{\prime})$ for $(A,b)$ and
$(A,b^{\prime})$ controllable (see Remark \ref{r-con}).
In particular we can define $\mathrm{RC}(A)$ as the value $\mathrm{RC}(A,b)$
corresponding to the case where $(A,b)$ is controllable,
and similarly for $\mathrm{RC}_{\#}(A)$, $\mathrm{RD}_{\#}(A)$, $\mathrm{RD}%
(A)$.

\begin{theorem}
Let $A$ be such that there exists $b$ with $(A,b)$ controllable. Then
\begin{equation}
\label{primaparte}\mathrm{RC}(A)=\mathrm{RD}(-A),
\end{equation}
and in particular $\mathrm{RC}(A)=+\infty$ if and only if $\mathrm{RD}%
(-A)=+\infty$. Moreover, there exists $c>0$ depending on $A$ but not on
$T,\mu$ such that if $\mathrm{RC}(A)>c$ or $\mathrm{RD}(A)>c$ then
$\mathrm{RC}(A)=\mathrm{RC}_{\#}(A)=\mathrm{RD}_{\#}(-A)=\mathrm{RD}(-A). $
\end{theorem}

\begin{proof}
According to Theorem~\ref{corri}, $\mathrm{PLARC}(A,b)$ is dense.
Theorem~\ref{5p1} then implies \eqref{primaparte}.

Let now $c>0$ be as in the statement of Theorem~\ref{th-rc}. Assume that
$\mathrm{RC}(A)>c$ (the case $\mathrm{RD}(A)>c$ being entirely similar). Hence
the set $\Xi_{c}=\{K\in M_{1,d}(\mathbb{R})\mid\mathrm{rc}(A,B,K)>c\}$ is
nonempty. By taking $\bar{\alpha}=1$ in \eqref{td-vp}, condition
$\mathrm{rc}(A,b,K)>c$ implies that the eigenvalues of $A+bK$ have real part
smaller than $-c$. Hence, Theorem~\ref{th-rc} implies that $\Xi_{c}%
\subset\mathrm{PLARC}(A,b)$. By equation~\eqref{4eq} one deduces that
$\mathrm{RC}(A),\mathrm{RC}_{\#}(A),\mathrm{RD}_{\#}(-A),\mathrm{RD}(-A)$ are
all equal.
\end{proof}

Combining the above result with equation \eqref{rem2} one gets the following corollary.

\begin{cor}
\label{cor-symm} Let $A$ be such that there exists $b$ with $(A,b)$
controllable and $-A+\frac{2}{d}\mathrm{Tr}(A)\mathrm{Id}_{d}$ is similar to
$A$. Then
\[
\mathrm{RC}(A)=\mathrm{RD}(A)-\frac{2}{d}\mathrm{Tr}(A),
\]
and in particular $\mathrm{RC}(A)=+\infty$ if and only if $\mathrm{RD}%
(A)=+\infty$.
\end{cor}

\begin{remark}
When restricted to the case $d=2$, Corollary~\ref{cor-symm} implies (and
actually improves) Proposition~\ref{rdc0}, established in
\cite{ChitourSigalotti2010}, because every traceless 2$\times2$ matrix $A$ is
similar to its opposite $-A$.
\end{remark}

\begin{remark}
{
Notice that a  matrix $A$ diagonalizable over $\C$ is similar to $-A$ if and only if
$\mathrm{Tr}(A^{k})=0$ for every odd integer $k$. In particular, every 
skew-symmetric matrix $A$ for which $(A,b)$ is controllable for some $b$ verifies
the hypotheses of Corollary~\ref{cor-symm} and we conclude that 
$\mathrm{RC}(A)=\mathrm{RD}(A)$ for such matrices.
The same is true for $J_d$, since every nilpotent matrix is similar to its opposite, and 
 even 
more can be established, as stated below.
}
\end{remark}

\begin{proposition}
Let $\mathcal{T}$ be equal to the diagonal matrix $\mathrm{diag}%
(1,-1,\dots,(-1)^{d+1})$. For every $K\in M_{1, d}(\mathbb{R})$ set
$K_{-}=(-1)^{d} K \mathcal{T} . $ Then $\mathrm{rc}(J_{d},e_{d},K)=\mathrm{rd}%
(J_{d}, e_{d}, K_{-})$ for every $K\in M_{1, d}(\mathbb{R})$ with components
in $\mathbb{R}\setminus\{0\}$.
\end{proposition}

\begin{proof}
Notice that
\begin{align*}
\mathcal{T}^{-1}  &  =\mathcal{T},\quad J_{d}=-\mathcal{T} J_{d} \mathcal{T},
\quad\mathcal{T} e_{d}=(-1)^{d+1} e_{d}.
\end{align*}

According to \eqref{chco1}, with $P=\mathcal{T}$, one gets
\[
\mathrm{rc}(J_{d},e_{d},K)=\mathrm{rc}(\mathcal{T} J_{d} \mathcal{T}%
^{-1},\mathcal{T} e_{n},K \mathcal{T}^{-1}) =\mathrm{rc}(- J_{d},(-1)^{d+1}
e_{d},(-1)^{d} K_{-}).
\]

Proposition \ref{acc} guarantees that $K_{-}\in\mathrm{PLARC}(J_{d},e_{d})$
and then Corollary~\ref{cooro} implies that
\[
\mathrm{rc}(- J_{d},(-1)^{d+1} e_{d}, (-1)^{d} K_{-})=\mathrm{rd}%
(J_{d},(-1)^{d} e_{d}, (-1)^{d} K_{-}).
\]
Since $\mathrm{rd}(A,b,L)=\mathrm{rd}(A,\xi b,\xi^{-1}L)$ for every
controllable pair $(A,b)$ and every $\xi\ne0$, we conclude.
\end{proof}

\section{Appendix}
In this appendix, we prove the following lemma which is used in the proof of Proposition \ref{last1}.

\begin{lemma}
The eigenvalues of any element of $\mbox{spin}(9,1)$ are symmetric with respect to the origin,  i.e., if 
$\lambda$ is an eigenvalue 
then $-\lambda$ is as well with the same algebraic multiplicity as $\lambda$.
\end{lemma}
\begin{proof}
Proving the result amounts showing that, for every $M\in \mbox{spin}(9,1)$,  the characteristic polynomial $P_M(X)$ of $M$ can be written as $P_M(X)=Q(X^2)$ where $Q$ is a unitary polynomial of degree five.

Recall that $\mbox{spin}(9,1)=\{M\in M_{10}(\mathbb{R}) \mid M^T\mathrm{Id}_{(9,1)}+\mathrm{Id}_{(9,1)}M=0\}$, where $\mathrm{Id}_{(9,1)}=\diag(\mathrm{Id}_9,-1)$. Rewrite $\mathrm{Id}_{(9,1)}=\mathrm{Id}_{10}-2e_{10}e_{10}^T$ with $e_{10}=(0\ \cdots\ 0\ 1)^T\in \mathbb{R}^{10}$. Let $M\in \mbox{spin}(9,1)$ and set  $v=M^Te_{10}$. Then, one immediately deduces that 
$\frac{M+M^T}2=ve_{10}^T+e_{10}v^T$ and then there exists a skew-symmetric matrix $A\in M_{10}(\mathbb{R})$ such that 
$$
M=A+ve_{10}^T+e_{10}v^T.
$$
By using the definition of $v$, one deduces from the above equality that $Ae_{10}=(v^Te_{10})e_{10}$. Since $A$ is skew-symmetric, one gets that $Ae_{10}=0$ and $v^Te_{10}=0$. Therefore, there exists a skew-symmetric matrix $A_1\in M_{9}(\mathbb{R})$ and $v_1\in \mathbb{R}^9$ such that 
$$
M=\begin{pmatrix}
A_{1} & v_1\\
v_1^T & 0%
\end{pmatrix}.
$$
Set $\alpha=\|v_1\|\geq 0$. Up to similarity with a matrix $U=\begin{pmatrix}
U_{1} & 0\\
0 & 1%
\end{pmatrix}$ so that $U_1\in SO(9)$ and $U_1^Tv_1=\alpha f_1$ with $f_1=(1\ 0\ \cdots\ 0)^T\in \mathbb{R}^{9}$, one can assume with no loss of generality that $v_1=\alpha f_1$. If $A_2$ is the $8\times 8$ skew-symmetric obtained from $A_1$ by removing the first line and the first column, one obtains after computations, that 
$$
P_M(X)=XP_{A_1}(X)+\alpha^2 P_{A_2}(X).
$$
Since $A_1$ and $A_2$ are skew-symmetric, one has that $P_{A_1}(X)$ and $P_{A_2}(X)$ are equal to $XQ_1(X^2)$ and $Q_2(X^2)$
respectively, where $Q_1$ and $Q_2$ are unitary polynomials of degree four and one concludes by setting $Q(X)=XQ_1(X)+\alpha^2Q_2(X)$.
\end{proof}

\bibliographystyle{abbrv}
\bibliography{biblio_CCS}

\end{document}